\documentclass[10pt,reqno]{amsart}
\usepackage{amssymb,amsmath,amsthm}
\usepackage{mathrsfs,pifont}
\usepackage{graphicx}
\usepackage{color}
\usepackage{float}
\usepackage{tabularx}
\usepackage{multirow}
\usepackage{amssymb,version,graphicx,subfigure,fancybox,mathrsfs,pifont,booktabs,wrapfig}
\usepackage{indentfirst}

    \numberwithin{equation}{section}
    \numberwithin{table}{section}
    \numberwithin{figure}{section}

\newtheorem{thm}{Theorem}[section]

\newtheorem{lemma}{Lemma}[section]
\newtheorem{rem}{Remark}[section]
\theoremstyle{definition}
\newtheorem{defn}{\bf Definition}[section]

\newcommand{\re}[1]{(\ref{#1})}

\def \ri {{\rm i}}

\def \af {\alpha}

\def \swf {\bs\psi^{\alpha,n}}

\def \bx{\bs x}

\def \bY{\bs Y}

\def\Div{\nabla\cdot}
\def\Curl{\nabla\times}

\def\e{\mathrm{e}}
\def\d{\mathrm{d}}
\def\RR{\mathbb{R}}

\def\sph{\mathbb{S}}
\def\ball{\mathbb{B}}

\def\s{\sigma}
\def\NN{\mathbb{N}}
\def\tr{\mathsf{t}}

\def\P{{\bs   P}}

\newcommand{\bs}[1]{\boldsymbol{#1}}
\begin{document}

\graphicspath{{figs/}}

\title[divergence free ball PSWFs]
{Vectorial ball Prolate spheroidal wave functions   \\
 with the divergence free constraint}
\author[
	J. Zhang,\; H. Li,
	]{
		\;\; Jing Zhang${}^1$, \;\; Guidoum Ikram${}^{1}$ \;\; Huiyuan Li${}^{2}$
		}
	\thanks{${}^1$School of Mathematics and Statistics \& Hubei Key Laboratory of Mathematical Sciences,  Central China Normal University, Wuhan 430079, China. The work of the first and the second authors  is partially supported by the National Natural Science Foundation of China (NSFC 11671166) and the Fundamental Research Funds for the Central Universities (CCNU19TS033, CCNU19TD010).\\
		\indent ${}^{2}$State Key Laboratory of Computer Science/Laboratory of Parallel Computing,  Institute of Software, Chinese Academy of Sciences, Beijing 100190, China. Email: huiyuan@iscas.ac.cn. The research of the third author is partially   supported by the National Natural Science Foundation of China (NSFC 11871455 and NSFC 11971016).\\
}




\begin{abstract}
 In this paper, we introduce one family of  vectorial  prolate spheroidal wave functions  of real order $\af>-1$ on the  unit ball in $\mathbb{R}^3$,
 which satisfy the divergence free constraint, thus are termed as divergence free vectorial ball PSWFs. They are  vectorial eigenfunctions of  an  integral operator related to the finite Fourier transform, and solve the  divergence free constrained maximum concentration problem in three dimensions, i.e.,  to what extent  can the total energy of a band-limited divergence free vectorial function  be concentrated on the unit ball?  Interestingly, any optimally concentrated  divergence free vectorial functions, when
represented in  series in  vector spherical harmonics,  shall be also concentrated in one of the three vectorial spherical harmonics modes. Moreover,
divergence free ball PSWFs are exactly the vectorial eigenfunctions of the second order Sturm-Liouville differential operator which  defines the scalar ball PSWFs.
Indeed,  the divergence free vectorial ball PSWFs possess a simple and close relation with the scalar ball PSWFs such that  they share the same merits.  Simultaneously, it turns out that the divergence free ball PSWFs solve  another  second order Sturm-Liouville eigen equation defined through the curl operator $\nabla\times $
instead of  the gradient operator $\nabla$.
\end{abstract}

\maketitle


\section{Introduction}
In the early 1960s, Slepian, Landau and Pollak answered an open question: to
what extent  are functions, which are confined to a finite bandwidth,
also concentrated in the time domain? (cf. \cite{Moo.M04,Slep61}). Any square integrable function $f(\xi)$  is bandlimited, if its Fourier transform $\psi(t)$ has a finite support
$[-c,c]$ such that
\begin{equation}\label{Fourierfinite}
f(\xi)=\int_{-1}^1 \psi(t)e^{\ri c\xi t}\,dt,\quad \xi\in(-\infty,\infty).
\end{equation}
The related issue is to what extent that the energy of such  $f(\xi)$ can be maximally concentrated on finite interval $I=(-1,1)$, that is,
\begin{align*}
 \max_{f} \Big\{{\displaystyle\int_{I} |f(\xi)|^2d\xi}\bigg/{\displaystyle\int_{\RR}
 |f(\xi)|^2\d{\xi}}\Big\}.
\end{align*}
The above problem is equivalent to
\begin{align*}
   \max_{\psi} \Big\{{\displaystyle \int_{I} \int_{I}  \frac{\sin c(x-t)}{\pi(x-t)}
 \psi(x)\cdot \overline{\psi(t)} \d{x} \d{t}}\bigg/{\displaystyle\int_{I} |\psi(t)| ^2\d t}\Big\}.
\end{align*}
It actually comes down to the study of the  following integral equation:
\begin{align}\label{Qpsi0}
\int_{I}\frac{\sin c(x-t)}{\pi(x-t)} \psi(t) \d {t} = \mu\,\psi(x),\;\;x\in I.
\end{align}
 Eigenfunctions, $\psi_n(x;c)$, $n=1,2,\dots$, of the integral equation \eqref{Qpsi0},   therein referred to as prolate
spheroidal wave functions (PSWFs),  are  discovered coincidentally to be  the eigenfunctions  of  an integral operator related to the finite Fourier transform:
\begin{equation}\label{pswf0int}
\lambda_n(c)\psi_n(x;c) =\int_{-1}^1 \e^{\ri cxt}\psi_n(t;c) \d t,\quad c>0,\;\;x\in I.
\end{equation}
 From this perspective,  PSWFs are initially defined as  the bandlimited functions most concentrated on the finite interval $I$. On the other hand, the PSWFs
 are exactly eigenfunctions of the second-order
singular Sturm-Liouville differential equation:
\begin{equation}\label{Pswfzero}
\partial_x \big((1-x^2)\partial_x \psi_n(x;c)\big)+\big(\chi_n(c)-c^2x^2\big)\psi_n(x;c)=0,\quad c>0,\quad x\in I,
\end{equation}
which naturally form an orthogonal basis of the $L^2$ space.

There has been abundant literature  addressing this research topic in more than 50 years past. (cf.\cite{Rokhlin.07,Rokhlin.12,Botezatu.16,Bonami.17,Landa.17}) Indeed, PSWFs of order zero and  multidimensional extensions (cf. \cite{Khalid.16, Simons.11,ZhangLi.18}) have
enjoyed applications in a wide range of science and engineering (cf. \cite{Jackson.91,Mathews.16,Thomson.76}). Notably, they are also well suited to approximate bandlimited functions, and have been proven to be a useful basis for spectral method, which  enjoy a much higher resolution for
highly oscillatory waves over the Legendre polynomial
based methods (cf. \cite{Boyed.04,Chen.05,ZhangWang.17}). These attractive properties have motivated us to use PSWFs as
basis functions in the study of  the acoustic wave equation  and of
the Maxwell system with large wave number, the two most common wave equations encountered in
physics or in engineering. It is well known that a physically realizable time-harmonic electromagnetic
field in a linear, isotropic, homogeneous medium must be divergence free (cf. \cite{Cockburn.05,Li.05,Brenner.07}).  One may ask whether there
are some kinds  of vectorial PSWFs that can be  optimally concentrated  within a given finite domain and  satisfy the divergence free constraint. i.e.,$\nabla\cdot {\bs\psi}(\bx)=0$.

To answer this question, we shall consider in this paper the concentration problem on the unit ball $\ball:=\big\{\bx\in \RR^3:\|\bx\|{ \leq}1\big\}$.
By extending the finite Fourier transform \eqref{Fourierfinite}  to  the one for vectorial functions on the unit ball $\ball,$
\begin{equation}\label{inbdlim}
\int_{\ball}\e^{-\ri c \langle \bx, \bs \tau\rangle}{\bs \psi}(\bs \tau;c)(1-|\bs \tau|^2)^{\af}\d\bs \tau=\lambda{\bs \psi}(\bx;c),\quad \bx\in\ball,\;\; c>0,\;\af>-1,
\end{equation}
we aim at  finding  some kinds
of  divergence free vectorial eigenfunctions of the above equation.
For this purpose, we first show that any divergence free vectorial function takes the form
\begin{equation*}
{\bs\psi}(\bx)=  \bx \times\nabla \phi(\bx)+ \nabla\times(\bx \times\nabla) \theta(\bx).
\end{equation*}
Further  we identify that all divergence free eigenfunctions of \eqref{inbdlim} are constituted only by  vectorial functions
 of  the form $\bx \times\nabla \phi(\bx)  $.
In the sequel,
with the  help of spherical harmonics $Y_{\ell}^n(\hat{\bx})$
in the spherical-polar coordinates $\bx = r\bs{\hat x}$ with $r\ge 0$ and $\bs{\hat x}\in \sph^{2}$, we  define the divergence free ball PSWFs as the band-limited vectorial functions
$$ \swf_{k,\ell}(\bx;c)= \bx \times\nabla [\phi^n_{\ell}(r;c) Y_\ell^{n}(\hat{\bx})]
=  \phi^n_{\ell}(r;c)\, \bx \times\nabla Y_\ell^{n}(\hat{\bx}), \quad 1\le\ell\le 2n+1,\,k,n\ge 0,$$ which satisfy the integral equation \eqref{inbdlim}
and are optimally concentrated on the ball.

 Recalling  that  vectorial spherical harmonics  fall into three types of  modes, $\hat{\bx} Y_{\ell}^n(\hat{\bx}) $, $r \nabla Y_{\ell}^n(\hat{\bx}) $ and $\bx \times\nabla  Y_\ell^{n}(\hat{\bx})  $,
we discover the interesting phenomenon that  any optimally concentrated   band limited and  divergence free
functions  will  certainly concentrate on the one of  the three mode types.

Simultaneously, it turns out that the divergence free ball PSWFs are exactly the eigenfunctions of the second order
Sturm-Liouville differential operator,
 \begin{equation}
 \begin{split}
\big[-(1-&\|\bx\|^2)^{-\af}\nabla\cdot \big((1-\|\bx\|^2)^{\af+1}\nabla\big)-\Delta_0+ c^2 \|\bx\|^2\big]{\bs \psi}(\bx; c)
\\
=\,&\chi\, {\bs \psi}(\bx; c),\qquad  \bx\in\ball, \  \alpha>-1,
\end{split}
\label{ball_pswf}
\end{equation}
where $\Delta_0 = (\bx \times \nabla) \cdot (\bx \times \nabla) $ is the Laplace-Beltrami operator.
This eigen-equation extends the differential property of the one dimensional PSWFs defined by  Slepian.
More astonishingly,   we find the divergence free ball PSWFs solve  the  following eigen-equation composed by the curl operator $\nabla\times $
instead of  the gradient operator $\nabla$,
 \begin{equation}
 \begin{split}
\big[(1-&\|\bx\|^2)^{-\af}\nabla\times \big((1-\|\bx\|^2)^{\af+1}\nabla\times \big)-\Delta_0+ c^2 \|\bx\|^2\big]{\bs \psi}(\bx; c)
\\
=\,&(\chi+2\alpha+2)\, {\bs \psi}(\bx; c),\qquad  \bx\in\ball, \  \alpha>-1.
\end{split}
\label{ball_pswf_curl}
\end{equation}

Moreover, we explore  their connections with scalar  ball PSWFs. The scalar  ball PSWFs, denoted by  $\psi_{k,\ell}^{\af,n}(\bx;c)$, $1\le \ell\le 2n+1,\,k,n\ge 0,$  are  bandlimited functions  share the same merit of divergence free vectorial ball PSWFs such that they are scalar eigenfunctions of the integral equation \eqref{inbdlim} and the differential equation \eqref{ball_pswf} (cf. \cite{ZhangLi.18}). As a result, in the spherical-polar coordinates, divergence free vectorial ball PSWFs
have the simple representation by the scalar ball PSWFs,
\begin{equation}\label{VPSWF}
\begin{split}
\swf_{k,\ell}(\bx;c)&=(\bx\times\nabla)\psi_{k,\ell}^{\af,n}(\bx;c), \quad 1\le  \ell\le 2n+1,\;\; k\in {\mathbb N}_0,\;n\in {\mathbb N}.
\end{split}
\end{equation}

We organise the remainder of the paper as follows. In Section \ref{Sect:2}, we collect some relevant properties of the  Jacobi Polynomials, spherical harmonics and
ball polynomials to be used throughout the paper. In Section \ref{sect3}, we define  the divergence free ball PSWFs as the vectorial eigenfunctions of the integral operators.
In Section \ref{sec4}, we study the  divergence free ball PSWFs as the vectorial eigenfunctions of the Sturm-Liouville differential equation on an  unit
ball and present their analytic properties. In Section \ref{sec5},  we describe an efficient method for computing the  divergence free ball PSWFs using the differential operator.
Flow field diagrams are presented to  illustrate  the geometrical properties of some  divergence free ball PSWFs.

\section{mathematical preliminary}\label{Sect:2}
In this section, we review the Jacobi Polynomials, and introduce the spherical harmonics and ball polynomials to facilitate the discussions in the forthcoming sections (cf.\cite{Dai2013,STW11}).
\subsection{Notations and vector calculus} We begin by introducing some conventions  and notations.
Denote by $\NN_0$ and $\NN$  the collection of nonnegative integers and  positive  integers, respectively.
 Let $\mathbb{R}^d$ ($d\in \NN$) be the  $d$-dimensional Euclidean space.
 Throughout this paper, we shall always use bold letters such as $\bx$ and $\bs y $ to denote column vectors.
For instance,  we write $\bx =
(x_1,x_2,\cdots,x_d)^{\tr}\in  \mathbb{R}^d$ as a column vector, where $(\cdot)^{\tr}$ denotes matrix or vector transpose. The inner product of $\bx,\bs y\in \mathbb{R}^d$ is denoted by $\bx\cdot\bs y$ or $\langle\bx,\bs y\rangle:=   \bx^{\tr} \bs y =\sum^d_{i=1} x_iy_i$, and the
norm of $\bx$ is denoted by $\|\bx\| := \sqrt{ \langle\bx, \bx\rangle}=\sqrt{\bx^{\tr}\bx}$.
The unit sphere $\mathbb{S}^{2}$ and the unit ball $\ball$ of $\mathbb{R}^3$ are respectively defined by
\begin{equation*}
\mathbb{S}^{2}:=\big\{\hat \bx\in \RR^3: \|\hat \bx\|=1\big\},\quad \ball:=\big\{\bx\in \RR^3: r=\|\bx\|{ \leq}1\big\}.
\end{equation*}
For   each $ \bx\in \RR^3$, we introduce its polar-spherical coordinates $(r,\hat \bx)$ such that $r=\|\bx\|$ and $ \bx =r\hat\bx:=r(\hat x_1,\hat x_2,\hat x_3)\in \mathbb{S}^{2}.$ Define the inner product of $L^2(\mathbb{S}^{2})$ as
\begin{equation*}
        ( f, g )_{\mathbb{S}^{2}}: = \int_{\mathbb{S}^{2}} f(\hat \bx ) g(\hat \bx) \d\s(\hat \bx),
\end{equation*}
where $d \s$ is the surface measure.

Define the spherical gradient operator $\nabla_0$ and the Laplace-Beltrami operator $\Delta_0$,
\begin{align}
\label{opert-1b}
& \nabla_0=\|\bx\|\, [\nabla-{\hat \bx}(\hat{\bx}\cdot\nabla)]=r\nabla - \bx\partial_r,
\\
\label{Lap-Bel}
& \Delta_0=\nabla_0 \cdot \nabla_0 = \|\bs x\|^2 \Delta - ({\bs x}\cdot\nabla)({\bs x}\cdot\nabla+1).
 \end{align}
 Indeed, $\nabla_0$ and $\Delta_0$ represent the spherical components of $\nabla$ and $\Delta$, respectively. Hence,
 \begin{align}
 \label{opert-1a} &\hat {\bs x} \cdot \nabla_0 = 0.
 \end{align}

Further, denote  by  ${\bs a} \times {\bs b}$  the cross product  of two vectors ${\bs a}, {\bs b} \in\RR^3$.
We   introduce the curl  operator ${\rm curl} = \nabla \times  $ and the divergence operator  ${\rm div} = \nabla \cdot $  for vectorial functions in $\RR^3$.
We are interested in the  vector calculus  involving
\begin{equation*}
\label{xnabla}
\bx \times \nabla =  ( x_2\partial_{x_3} -x_3\partial_{x_2},  x_3\partial_{x_1} -x_1\partial_{x_3}, x_1\partial_{x_2} -x_2\partial_{x_1} )^{\tr},
\end{equation*}
which is closely related to $\nabla_0$ and $\Delta_0$. The following two lemmas on $\bx \times \nabla$ will be used frequently  in the paper.
Their proofs  will be postponed to  Appendix \ref{AppFirst}.

\begin{lemma}
\label{calculus}
It holds that
\begin{subequations}
\begin{align}
\label{opert-1D}& {\bs x}\times \nabla =  -\nabla \times {\bs x} =\hat {\bs x}\times \nabla_0,
\\
\label{opert-1d}
& \nabla \cdot( {\bs x}\times \nabla) = {\bs x} \cdot( {\bs x}\times \nabla)= 0,
\\
\label{opert-1j}
&   ({\bx}  \times  \nabla) \cdot ( {\bs x}\times \nabla)=\Delta_0 ,
\\
\label{D0xn}
&\Delta_0(\bx\times\nabla)  = (\bx\times\nabla)\Delta_0.
\end{align}
\end{subequations}
\end{lemma}

\begin{lemma}
\label{calculus2}
It holds that
\begin{subequations}
\begin{align}
\label{opert-1g} & \bx \times( {\bs x}\times \nabla) = \bs x  (\bs x \cdot \nabla ) -\|\bx \|^2\, \nabla = -\|\bx \|\,\nabla_0,
\\
\label{opert-1h}
 & \nabla \cdot (\bx \times( {\bs x}\times \nabla)) = -\Delta_0   ,
\\
\label{opert-1E} &\nabla \times( {\bs x}\times \nabla) = \bs x \Delta - (\bs x \cdot \nabla +2) \nabla,
\\
\label{opert-1i}
&\nabla \cdot(\nabla \times( {\bs x}\times \nabla))=0 ,
\\
\label{opert-1c}
&\bx\cdot(\nabla \times( {\bs x}\times \nabla))=\Delta_0.
\end{align}
\end{subequations}
\end{lemma}

The following lemma is a direct consequence of Proposition 1.8.4 in  \cite{Dai2013}.
\begin{lemma}\label{lemInbyP}
For $f,g\in C^1(\sph^{d-1})$ and $1\leq i\neq j\leq d$,
\begin{equation}\label{InbyP}
\int_{\sph^{d-1}}f(\hat  \bx) [\hat \bx \times \nabla_0 g(\hat \bx)]d\sigma(\hat \bx)=-\int_{\sph^{d-1}} [\hat \bx \times \nabla_0 f(\hat \bx)] g(\hat \bx)d\sigma(\hat \bx).
\end{equation}
\end{lemma}
%
\subsection{Jacobi Polynomials} We now briefly review some relevant properties of Jacobi Polynomials.
For real $\af,\beta>-1$, the normalized Jacobi polynomials, denoted by $\{J_{k}^{(\af,\beta)}(\eta)\}_{k\ge 0},$  are orthonormal with respect to the
Jacobi weight function $\omega^{\af,\beta}(\eta)=(1-\eta)^{\af}(1+\eta)^{\beta}$ over $I:=(-1,1)$,
\begin{equation}\label{Jacobiorth}
\int_{-1}^{1} {J}_{k}^{(\af,\beta)}(\eta){J}_{l}^{(\af,\beta)}(\eta)\omega_{\af,\beta} (\eta)
\d{\eta}=2^{\alpha+\beta+2}\delta_{kl}.
\end{equation}
They satisfy  the
three-term recurrence relation:
\begin{equation}\label{Jacobi}
\begin{split}
&\eta {J}_k^{(\af,\beta)}(\eta)=a_k^{(\af,\beta)} {J}_{k+1}^{(\af,\beta)}(\eta)+b_k^{(\af,\beta)} {J}_{k}^{(\af,\beta)}(\eta)+a_{k-1}^{(\af,\beta)} {J}_{k-1}^{(\af,\beta)}(\eta),
\\
& J_{0}^{(\af,\beta)}(\eta)=\frac{1}{h^{(\alpha,\beta)}_0},
\quad
J_{1}^{(\af,\beta)}(\eta)=  \frac{1}{2h^{(\alpha,\beta)}_1}
  \big( (\af+\beta+2)\eta+(\af-\beta)\big),
\end{split}
\end{equation}
where $\eta\in I$, and
\begin{align*}
&a_k^{(\af,\beta)}=\sqrt{ \frac{4(k+1)(k+\alpha+1)(k+\beta+1)(k+\alpha+\beta+1)}{(2k+\alpha+\beta+1)(2k+\alpha+\beta+2)^2(2k+\alpha+\beta+3) } }
,
\\
&b_k^{(\af,\beta)} = \frac{\beta^2-\alpha^2}{(2k+\alpha+\beta)(2k+\alpha+\beta+2)},
  \\
  & h_k^{(\af,\beta)} = \sqrt{\frac{\Gamma(k+\alpha+1)\Gamma(k+\beta+1)}{2(2k+\alpha+\beta+1) \Gamma(k+1) \Gamma(k+\alpha+\beta+1)}   }.
\end{align*}
The Jacobi polynomials are the eigenfunctions of the
Sturm-Liouville problem
\begin{equation*}
\mathscr{L}_{\eta}^{(\af,\beta)} J_{k}^{(\af,\beta)}(\eta):=
-\frac{1}{\omega_{\af,\beta}(\eta)}
\partial_{\eta}\big(\omega_{\af+1,\beta+1}(\eta)\partial_{\eta}J_{k}^{(\af,\beta)}(\eta)\big)=
 \lambda_k^{(\af,\beta)}J_k^{(\af,\beta)}(\eta),\quad  \eta\in I,
\end{equation*}
and the corresponding  eigenvalues are $\lambda_k^{(\af,\beta)}=k(k+\af+\beta+1).$

\subsection{Spherical harmonics and ball polynomials}
We first  define  the trivariate  harmonic polynomials  of total degree $n\in \NN_0$ through  the spherical coordinates
$\bx=(r\sin \theta \cos \phi, r\sin \theta \sin \phi, r\cos \theta)^{\tr}$,
\begin{align*}
 Y^n_1(\bx) = \frac{r^n}{\sqrt{8\pi}}J^{(0,0)}_{n} (\cos \theta ), \qquad  Y^n_{2\ell}(\bx) &= \frac{ r^n}{2^{\ell+1} \sqrt{ \pi}} (\sin\theta)^{\ell}  J^{(\ell,\ell)}_{n-\ell} (\cos \theta ) \cos \ell \phi, \quad   1\le \ell \le n,
 \\
Y^n_{2\ell+1}(\bx)  &=  \frac{ r^n}{2^{\ell+1} \sqrt{ \pi}}  (\sin\theta)^{\ell}  J^{(\ell,\ell)}_{n-\ell} (\cos \theta ) \sin \ell \phi, \quad 1\le  \ell \le  n.
\end{align*}

Indeed,  a polynomial  $Y$ is referred to as a harmonic polynomial of (total) degree $n$ if  $Y$ is homogeneous of degree $n$ and satisfies  the harmonic equation  (cf. \cite{Dai2013}),
\begin{equation*}
 Y(\bx) = r^n Y(\hat \bx), \qquad   \Delta Y(\bx) =0.
\end{equation*}

Obviously,  harmonic polynomials  $Y_\ell^n$, $ 1\le \ell\le 2n+1,\, n\in \NN_0$,  are uniquely determined by their restrictions on the unit sphere, $Y_\ell^n |_{\mathbb{S}^2}$,
$ 1\le \ell\le 2n+1,\, n\in \NN_0$,  while the laters  are   exactly   the well-known spherical harmonics.
Hereafter for notational convenience,  we introduce the index set
 \begin{equation*}
 \Upsilon_{\!0}=\{(\ell,n)\in \NN_0^2\,:\, 1\le\ell\le 2n+1\},\qquad  \Upsilon=\{(\ell,n)\in \NN^2\,:\, 1\le\ell\le 2n+1\}.
 \end{equation*}
With a little abuse, we shall use the same notation $Y_\ell^n$   both for a  harmonic polynomial and for  its corresponding spherical harmonic function.

In spherical polar coordinates, the Laplace operator can be written as
\begin{equation*}
   \Delta = \frac{d^2}{d r^2} + \frac{2}{r} \frac{d}{dr} + \frac{1}{r^2} \Delta_0.
\end{equation*}
Thus, the spherical
harmonics are  eigenfunctions of the Laplace-Beltrami operator,
\begin{equation} \label{eq:LaplaceBeltrami}
        \Delta_0  Y_\ell^n(\hat \bx) = - n(n+1) Y_\ell^n(\hat \bx), \qquad  (\ell,n) \in \Upsilon_{\!0},
\end{equation}
In view of \eqref{eq:LaplaceBeltrami} and \eqref{Jacobiorth}, we have the orthogonality:
\begin{equation*}
 ( Y_\ell^n, Y_\iota^m)_{\mathbb{S}^{2}}=\delta_{nm} \delta_{\ell\iota},
 \quad  (\ell,n),(\iota,m) \in \Upsilon_{\!0}.
 \end{equation*}

With the spherical harmonics, for any $\alpha>-1$, we define the ball polynomials as
\begin{equation*}
P_{k,\ell}^{\af,n}(\bx)={J}_{k}^{(\af,n+\frac{1}2)}(2\|\bx\|^2-1) Y_{\ell}^{n}(\bx),
\quad \bx\in \ball, \;\;  (\ell,n) \in \Upsilon_{\!0},\;k\in {\mathbb N}_0.
\end{equation*}
Note that the total degree of  $P_{k,\ell}^{\af,n}(\bs x)$ is $n+2k$.
The ball polynomials are mutually orthogonal
with respect to the weight function $\varpi_{\af}(\bx):=(1-\|\bx\|^2)^{\af}$ (cf.  {\cite[Propostion 11.1.13]{Dai2013}}):
\begin{equation*}
(P_{k,\ell}^{\af,n},P_{j,\iota}^{\af,m} )_{\varpi_{\af}} =\delta_{nm}\delta_{kj}\delta_{\ell \iota},
 \quad (\ell,n),(\iota,m) \in \Upsilon_{\!0},\;\; k,j\in {\mathbb N}_0,
\end{equation*}
where the inner product $(\cdot,\cdot)_{\varpi_{\af}} $ is defined by
\begin{equation*}
(f,g)_{\varpi_{\af}}:=\int_{\ball}f(\bx)g(\bx)\varpi_{\af}(\bx)\,\d\bx.
\end{equation*}


\begin{lemma}[{\cite[Theorem 11.1.5]{Dai2013}}]\label{dxbop} The ball orthogonal polynomials 
are the eigenfunctions of the differential operator:
\begin{equation*}
\mathscr{L}_{\bx}^{(\af)}P_{k,\ell}^{\af,n}(\bs x):=\left(-\Delta+\nabla \cdot \bx(2\af+\bx\cdot\nabla)-6\af\right) P_{k,\ell}^{\af,n}(\bs x)=\gamma_{n+2k}^{(\af)}P_{k,\ell}^{\af,n}(\bs x),
\end{equation*}
where $\gamma_m^{(\af)}:=m(m+2\af+3).$
\end{lemma}
The Sturm-Liouville operator $\mathscr{L}_{\bs x}^{(\af)}$ takes different
forms, which serve as
preparations for the study of vectorial ball polynomials in the forthcoming sections.
\begin{thm}[{\cite[Theorem 2.2]{Dai2013}}]\label{thmdx}
For $\alpha>-1$, it holds that
\begin{equation}
\begin{split}
\mathscr{L}_{\bx}^{(\af)} =& -(1-\|\bx\|^2)^{-\af}\nabla\cdot({\bs {\rm I}}-\bx\bx^{\tr})(1-\|\bx\|^2)^{\af}\nabla
\\
=& -(1-\|\bx\|^2)^{-\af} \nabla\cdot  (1-\|\bx\|^2)^{\af+1}\nabla - \Delta_0
\\
=&-(1-r^2)\partial^2r-\frac{d-1}{r}\partial r+(2\af+4)r\partial r-\frac{1}{r^2} \Delta_0,
\end{split}
\label{Ldef}
\end{equation}
where $\Delta_0$ is the spherical part of $\Delta$ and involves only derivatives in $\hat \bx.$
\end{thm}
\subsection{Vector spherical harmonics.}\label{secP}
Now, we  introduce the definition of vectorial spherical harmonics (cf. \cite{Bar85}),
\begin{subequations}
\begin{align}
\label{y-1}&\bY_{\ell}^{n,1}(\hat{\bx})=\hat{\bx}Y_{\ell}^{n}(\hat{\bx}), \\
\label{y-2}&\bY_{\ell}^{n,2}(\hat{\bx})=r\nabla Y_{\ell}^{n}(\hat{\bx})=\nabla_0 Y_{\ell}^{n}(\hat{\bx}),\\
\label{y-3}&\bY_{\ell}^{n,3}(\hat{\bx})=\bx\times\nabla Y_{\ell}^{n}(\hat{\bx})=\hat{\bx}\times\nabla_0 Y_{\ell}^{n}(\hat{\bx}).
\end{align}
\end{subequations}
and the relations among
scalar and vector spherical harmonics (cf. \cite{Bar85}):
\begin{subequations}
\begin{align}
\label{div-1}  & \nabla\cdot(f(r) {\bs Y}_\ell^{n,1}(\hat{\bs x}))=(\partial_r+\frac{2}{r})f(r) Y_\ell^{n}(\hat{\bs x}), \\
\label{div-2} & \nabla\cdot(f(r) {\bs Y}_\ell^{n,2}(\hat{\bs x}))=-\frac{n(n+1)}{r} f(r) Y_\ell^{n}(\hat{\bs x}),\\
\label{div-3}&  \nabla\cdot(f(r) {\bs Y}_\ell^{n,3}(\hat{\bs x}))=0,
\end{align}
\end{subequations}
and
\begin{subequations}
\begin{align}
\label{curl-1} &\nabla\times(f(r) {\bs Y}_\ell^{n,1}(\hat{\bs x} ))=-\frac{1}{r}f(r) {\bs Y}_\ell^{n,3}(\hat{\bs x} ), \\
\label{curl-2} & \nabla\times(f(r) {\bs Y}_\ell^{n,2}(\hat{\bs x} ))=(\partial_r+\frac{1}{r})f(r) {\bs Y}_\ell^{n,3}(\hat{\bs x} ), \\
\label{curl-3} & \nabla\times(f(r) {\bs Y}_\ell^{n,3}(\hat{\bs x} )) = -\frac{n(n+1)}{r}f(r) {\bs Y}_\ell^{n,1}(\hat{\bs x} ) - (\partial_r+\frac{1}{r})f(r){\bs Y}_\ell^{n,2}(\hat{\bs x} ).
\end{align}
\end{subequations}
Moreover, it can be checked that that the vector spherical harmonics are orthogonal in
the same sense as the spherical harmonics,
\begin{eqnarray*}
 {\bs Y}_\ell^{n,1}(\hat{\bs x} )\cdot {\bs Y}_\ell^{n,2}(\hat{\bs x} )=0, \qquad {\bs Y}_\ell^{n,1}(\hat{\bs x} )\cdot {\bs Y}_\ell^{n,3}(\hat{\bs x} )=0, \qquad {\bs Y}_\ell^{n,2}(\hat{\bs x} )\cdot {\bs Y}_\ell^{n,3}(\hat{\bs x} )=0.
\end{eqnarray*}
Moreover, the following orthogonality  holds  for  $(\ell,n),(\iota,m) \in \Upsilon_{\!0}$ and $ 1\le i,j\le 3 $,
\begin{align*}
\int_{\mathbb{S}^{2}}{\bs Y}_\ell^{n,i}(\hat{\bs x} )\cdot {\bs Y}_\iota^{m,j}(\hat{\bs x} ) \d\s(\hat \bx) = [ \delta_{i,1}+ n(n+1) ( 1- \delta_{i,1})]
\delta_{ij}\delta_{\ell\iota}\delta_{nm}.
\end{align*}

Thanks to the above lemma, it is then straightforward to prove the following result.
\begin{thm}\label{deY3} The vector spherical harmonics $\bY_{\ell}^{n,3}(\hat{\bx})$ are eigenfunctions of $\Delta_0,$
 \begin{equation*}
\Delta_0\bY_{\ell}^{n,3}(\hat{\bx})=-n(n+1)\bY_{\ell}^{n,3}(\hat{\bx}),\quad \forall\; \hat{\bx}\in\mathbb{S}^{2}.
\end{equation*}
\end{thm}
\begin{proof}

Thanks to \eqref{D0xn}, we have
\begin{equation*}
\begin{split}
\Delta_0\bY_{\ell}^{n,3}(\hat{\bx})&\overset{\re{y-3}}=\Delta_0(\bx\times\nabla)Y_{\ell}^{n}(\hat{\bx})\overset{\re{D0xn}}=(\bx\times \nabla)\Delta_0Y_{\ell}^{n}(\hat{\bx})\\
&\overset{\eqref{eq:LaplaceBeltrami}}=-n(n+1)(\bx\times \nabla)Y_{\ell}^{n}(\hat{\bx})\overset{\re{y-3}}=-n(n+1)\bY_{\ell}^{n,3}(\hat{\bx}).
\end{split}
\end{equation*}
This gives the proof.
\end{proof}
\section{Divergence free ball PSWFs  as  vectorial eigenfunctions of  finite Fourier transform}\label{sect3}
In this section, one may answer this question: whether there
are some kinds  of  band-limited vectorial functions that can be optimally spatially-concentrated
within a given spatial domain and satisfy the divergence free constraint.

\subsection{Scalar ball PSWFs}
Let us first review briefly the  scalar version of this question in arbitrary dimensions.
The optimal concentration problem for scalar functions  is  shown to be closely related to  prolate spheroidal wave functions (PSWFs) of real order $\af>-1$  on the unit ball  (cf.  {\cite{ZhangLi.18}}), which are the eigenfunctions of  a compact (finite) Fourier integral operator ${\mathscr F}_{c}^{(\af)}:{L^2_{\varpi_{\alpha}}(\ball)}\rightarrow
{L^2_{\varpi_{\alpha}}(\ball)}$, defined by
\begin{equation}\label{bdlim}
{\mathscr F}_{c}^{(\af)}[\phi](\bx)=\int_{\ball}\e^{-\ri
c \langle \bx, \bs \tau\rangle}\phi(\bs \tau)\varpi_{\alpha}(\bs \tau)\d\bs \tau,\quad \bx\in\ball,\;\; c>0,\;\af>-1.
\end{equation}

\begin{defn}\label{BPSWFs} {\bf (Ball PSWFs).} For real $\alpha>-1$ and real $c\ge 0,$
the prolate spheroidal wave functions
on a $d$-dimensional unit  ball $\ball^d,$ denoted by $\big\{\psi^{\af,n}_{k,\ell}(\bx; c)\big\}_{(\ell,n)\in \Upsilon_{\!0}}^{k\in {\mathbb N}_0},$ are  eigenfunctions of the integral operator $\mathscr{F}_{c}^{(\af)}$ defined in \eqref{bdlim},
\begin{equation}\label{eigen1}
{\mathscr F}_c^{(\alpha)}[\psi^{(\af,n)}_{k,{\ell}}](\bx;c)={(-\ri)}^{n+2k}\lambda_{n,k}^{(\af)}(c)\, \psi^{(\af,n)}_{k,{\ell}}(\bx;c),\quad
\bx\in \ball,
\end{equation}
where $c$ is the bandwidth parameter and the   modulus of  eigenvalues
 $\big\{\lambda_{n,k}^{(\af)}(c)\big\}_{k,n\in {\mathbb N}_0}$  are arranged for fixed $n$ as
\begin{equation*}\label{orderth}
\lambda_{n,0}^{(\af)}(c)>\lambda_{n,1}^{(\af)}(c)>\cdots>\lambda_{n,k}^{(\af)}(c)>\cdots>0.\;
\end{equation*}

\end{defn}

Note that for $\af= 0$, ${\mathscr F}_{c}^{(0)}$ is reduced to the finite Fourier transform on the  ball.
And $\psi^{0,n}_{k,\ell}(\bx; c)$ are the band-limited functions  most concentrated on the unit ball.

We then define the associated integral operator
${\mathcal Q}_c^{(\af)}: {L^2_{\varpi_{\alpha}}(\ball)}\rightarrow {L^2_{\varpi_{\alpha}}(\ball)},$ defined by
\begin{equation*}
{\mathcal Q}_c^{(\af)}=({{\mathscr F}_c^{(\alpha)}})^{*}\circ {\mathscr F}_c^{(\alpha)},\quad
 c>0,\;\af>-1.
 \end{equation*}
One verifies that
\begin{lemma}[{\cite[Theorem 4.1]{ZhangLi.18}}]
\label{th:Q}
Let $c>0,\af>-1$ and $\phi\in L_{\varpi_{\af}}^2(\ball).$ Then we have
\begin{equation*}
{\mathcal Q}_c^{(\af)}\big[\phi\big](\bx)=\int_{\ball} {\mathcal K}_c^{(\af)}(\bx,\bs \tau)\phi(\bs \tau){\varpi_{\af}}(\bs \tau)\d\bs \tau,\quad
\bx\in \ball,
\end{equation*}
where
\begin{equation*}\label{propOpQ}
{\mathcal K}_c^{(\af)}(\bx,\bs \tau):
=\frac{(2\pi)^{\frac{3}{2}}}{(c\|\bs \tau-\bx\|)^{\frac{1}{2}}}\int_{0}^1s^{\frac{3}{2}}(1-s^2)^{\af}J_{\frac{1}{2}}(cs\|\bs \tau-\bx\|) \d{s}.
\end{equation*}
\end{lemma}
The following theorem indicates  that the ball PSWFs  are  eigenfunctions of ${\mathcal F}_c^{(\alpha)}$ and ${\mathcal Q}_c^{(\alpha)}$
simultaneously.
\begin{lemma}[{\cite[Theorem 4.1]{ZhangLi.18}}]\label{lamth}For  $\alpha>-1$ and $c>0,$  the  ball PSWFs
$\big\{\psi^{(\af,n)}_{k,{\ell}}(\bx;c)\big\}_{(\ell,n)\in \Upsilon_{\!0}}^{k\in {\mathbb N}_0}$ are also the eigenfunctions
of ${\mathcal Q}_{c}^{(\af)}:$
\begin{equation*}\label{Qc}
{\mathcal Q}_{c}^{(\af)}[\psi^{(\af,n)}_{k,{\ell}}](\bx;c)=\mu_{n,k}^{(\af)}(c)\,\psi^{(\af,n)}_{k,{\ell}}(\bx;c),
\end{equation*}
and the eigenvalues satisfy
\begin{equation*}\label{mulambda}
\mu_{n,k}^{(\alpha)}(c)=
|\lambda_{n,k}^{(\af)}(c)|^2\,.
\end{equation*}
\end{lemma}

\subsection{Divergence of the finite Fourier transform of a divergence free field}
To solve the  optimal concentration problem for band-limited and  divergence free vector fields,
 it is crucial to choose  fields ${\bs E}(\bx;c)$ such that:
\begin{itemize}
\item  They are vectorial eigenfunctions of the (finite) Fourier integral operator,
       i.e., \begin{equation}
       \label{fcbphi}
       {\mathscr F}_{c}^{(\af)}[{\bs E}](\bx;c)=\lambda{\bs E}(\bx;c).\end{equation}
\item  They satisfy the divergence free constraint, i.e.,$\nabla\cdot {\bs E}(\bx;c)=0$.\\
\end{itemize}

To this end, we first represent a vector field  ${\bs E}\in L^2(\RR^3)^3$  as  a series in vector spherical harmonics under  spherical coordinates,
\begin{equation}
\label{E}
     {\bs E}(\bx) =\sum\limits_{n=0}^{\infty}\sum\limits_{\ell=1}^{2n+1}\big[E_{\ell}^{n,1}(r)\bY_\ell^{n,1}(\hat{\bs x})+E_{\ell}^{n,2}(r)\bY_\ell^{n,2}(\hat{\bs x})+E_{\ell}^{n,3}(r)\bY_\ell^{n,3}(\hat{\bs x})\big].
\end{equation}
The divergence of  ${\bs E}$ can be directly obtained  from \eqref{div-1}-\eqref{div-3},
\begin{eqnarray*}
\begin{aligned}
  \nabla\cdot {\bs E} =\sum\limits_{n=0}^{\infty}\sum\limits_{\ell=1}^{2n+1}\Big[  (\partial_r+\frac{2}{r})E_{\ell}^{n,1}(r)-\frac{n(n+1)}{r} E_{\ell}^{n,2}(r) \Big]Y_\ell^{n}(\hat{\bs x}).
\end{aligned}
\end{eqnarray*}
Hence, the divergence free constraint $\nabla\cdot {\bs E}=0$ is equivalent to
\begin{align}
\label{Edivfree}
E_{1}^{0,1}(r) = \frac{C}{r^2}, \quad \text{ and } \quad    E_{\ell}^{n,2}(r) =\frac{1}{n(n+1)} (\partial_r+\frac{1}{r}) [r E_{\ell}^{n,1}(r)], \quad (\ell,n) \in \Upsilon.
\end{align}
Then  \eqref{E}, \eqref{Edivfree}  and \eqref{curl-3}
state that a divergence free field ${\bs E}\in L^2(\RR^3)^3$ or ${\bs E}\in L^2(\mathbb{B})^3$  is necessarily   taking the form,
\begin{align}
\begin{split}
{\bs E}&\,=\sum\limits_{n=1}^{\infty}\sum\limits_{\ell=1}^{2n+1} \Big[ E_{\ell}^{n,3}(r) \bY_{\ell}^{n,3}(\hat\bx)
-\frac{1}{n(n+1)} \nabla\times\big(rE_{\ell}^{n,1}(r)  \bY_{\ell}^{n,3}(\hat \bx)\big) \Big]
\\
&\overset{\eqref{y-3}}=\sum\limits_{n=1}^{\infty}\sum\limits_{\ell=1}^{2n+1}
\Big[ (\bx \times \nabla ) \big(E_{\ell}^{n,3}(r) Y_{\ell}^{n}(\hat \bx)\big)
-   \frac{1}{n(n+1)} \nabla\times (\bx \times \nabla ) \big(rE_{\ell}^{n,1}(r)  Y_{\ell}^{n}(\hat \bx)\big) \Big]
\\
&:= \sum\limits_{n=1}^{\infty}\sum\limits_{\ell=1}^{2n+1} \big[ {\bs E}_{\ell}^{n,3}(\bx) + {\bs E}_{\ell}^{n,1}(\bx)\big].
\end{split}
\label{DFconst}
\end{align}

In the forthcoming discussion, we shall visit the divergence of the finite Fourier transform of a divergence free vectorial  field.
By resorting to the spherical-polar coordinates
$\bx = r\bs{\hat x}$ and $\bs \tau = \tau \bs{\hat \tau}$
 with $r,\tau\ge 0$ and $\bs{\hat x},\bs{\hat \tau}\in \sph^{2}$,  we deduce that
\begin{equation*}
\begin{split}
  [ {\mathscr F}_{c}^{(\af)} & {\bs E}_{\ell}^{n,3} ](\bx)
  =\int_{\ball}\e^{-\ri c \langle \bx, \bs \tau\rangle} {\bs E}_{\ell}^{n,3} ({\bs \tau})\omega_{\alpha}(\bs \tau)\d\bs \tau
  \\
\overset{\eqref{y-1}}=&\,  \int_0^1 (1-\tau^2)^{\alpha}  \tau^{2} E^{n,3}_{\ell} (\tau)\d{\tau}\
\int_{\sph^{2}} \e^{-\ri  c\tau r \langle {\bs{ \hat{x}}}, \hat {\bs \tau} \rangle}(\hat{\bs\tau}\times\nabla_{0,\bs \tau}) Y^{n}_{\ell}(\hat{\bs \tau})
 \d\s({\hat {\bs \tau}})
 \\
\overset{\eqref{InbyP}}=&\,  -\int_0^1 (1-\tau^2)^{\alpha}  \tau^{2}   E^{n,3}_{\ell}(\tau)  \d{\tau}\
\int_{\sph^{2}}[(\hat{\bs\tau}\times\nabla_{0,\bs \tau}) \e^{-\ri  c\tau r \langle {\bs{ \hat{x}}}, \hat {\bs \tau} \rangle}]Y^{n}_{\ell}(\hat{\bs \tau})
 \d\s({\hat {\bs \tau}})
 \\
=&-\int_{\ball}[(\bs\tau\times\nabla_{\bs \tau})\e^{-\ri
c \langle \bx, \bs \tau\rangle}]  E^{n,3}_{\ell}(\tau) Y^{n}_{\ell}(\hat{\bs \tau})\omega_{\alpha}(\bs \tau)\d\bs \tau
\\
=&\int_{\ball}[(\bx\times\nabla_{\bx})\e^{-\ri
c \langle \bx, \bs \tau\rangle}] E^{n,3}_{\ell} (\tau)Y^{n}_{\ell}(\hat{\bs \tau})\omega_{\alpha}(\bs \tau)\d\bs \tau
\\
=\,&(\bx\times\nabla)\int_{\ball}\e^{-\ri
c \langle \bx, \bs \tau\rangle} E^{n,3}_{\ell} (\tau)Y^{n}_{\ell}(\hat{\bs \tau})\omega_{\alpha}(\bs \tau)\d\bs \tau,
\end{split}\label{Freduc}
\end{equation*}
where  $\nabla_{\bs\tau}$ (resp. $\nabla_{\bx}$) is understood as the gradient operator  with respect to the variable
$\bs\tau$ (resp. $\bx$).
In view of \eqref{opert-1d},   the finite Fourier transform of $ {\bs E}_{\ell}^{n,3}(\bx)$ is  always divergence free,
\begin{align}
\label{divfree:E3}
[ \nabla \cdot  {\mathscr F}_{c}^{(\af)} &{\bs E}_{\ell}^{n,3} ](\bx)  = 0.
\end{align}
Meanwhile,
\begin{equation*}
\begin{split}
 [\nabla \cdot  {\mathscr F}_{c}^{(\af)} & {\bs E}_{\ell}^{n,1} ](\bx) =
-\frac1{n(n+1)}\nabla \cdot\int_{\ball}\nabla_{\bs \tau}\times (\bs \tau\times\nabla_{\bs \tau}) \big[\tau E_{\ell}^{n,1}(\tau )  Y_{\ell}^{n}(\hat{\bs \tau}) \big] \e^{-\ri
c \langle \bx, \bs \tau\rangle}\omega_{\alpha}(\bs \tau)\d\bs \tau
\\
=&-\frac1{n(n+1)}\int_{\ball}\nabla_{\bs \tau}\times(\bs \tau\times\nabla_{\bs \tau}) \big[\tau E_{\ell}^{n,1}(\tau )  Y_{\ell}^{n}(\hat{\bs \tau}) \big]  \cdot(-\ri c)\bs \tau\e^{-\ri
c \langle \bx, \bs \tau\rangle}\omega_{\alpha}(\bs \tau)\d\bs \tau
\\
\overset{\eqref{opert-1c}}=&\frac{\ri c}{n(n+1)}\int_{\ball}\Delta_{0,{\bs\tau}} \big[\tau E_{\ell}^{n,1}(\tau )  Y_{\ell}^{n}(\hat{\bs \tau}) \big]  \e^{-\ri
c \langle \bx, \bs \tau\rangle}\omega_{\alpha}(\bs \tau)\d\bs \tau
\\
\overset{\eqref{eq:LaplaceBeltrami}}=&(-\ri c)\int_{\ball} \tau E_{\ell}^{n,1}(\tau )  Y_{\ell}^{n}(\hat{\bs \tau}) \e^{-\ri
c \langle \bx, \bs \tau\rangle}\omega_{\alpha}(\bs \tau)\d\bs \tau.
\end{split}
\end{equation*}
Owing to the compactness of the finite Fourier transform operator ${\mathscr F}_{c}^{(\af)}$,
\begin{align}
\label{divfree:E1}
[\nabla \cdot  {\mathscr F}_{c}^{(\af)}  {\bs E}_{\ell}^{n,1} ](\bx)=0  \quad \text{ is equivalent to }\quad
 E^{n,1}_{\ell}(r)  = 0 =  E^{n,2}_{\ell}(r).
\end{align}
Since the finite Fourier transform of any divergence free vectorial eigenfunction of \eqref{fcbphi} is necessarily divergence free, the equations \eqref{divfree:E3} and \eqref{divfree:E1} reveal that
 a divergence free eigenfunction of \eqref{fcbphi} can only be  expressed as
\begin{align*}
\begin{split}
{\bs E}(\bx;c)&\,= \sum\limits_{n=1}^{\infty}\sum\limits_{\ell=1}^{2n+1} E_{\ell}^{n,3}(r;c) \bY_{\ell}^{n,3}(\hat\bx)
=\bx \times\nabla  \sum\limits_{n=1}^{\infty}\sum\limits_{\ell=1}^{2n+1} E_{\ell}^{n,3}(r;c) Y_{\ell}^{n}(\hat\bx)
\\
&:=\bx \times\nabla  \sum\limits_{n=1}^{\infty}\sum\limits_{\ell=1}^{2n+1} \phi^n_{\ell} (\bx,c).
\end{split}
\end{align*}
Thus  the eigenvalue problem \eqref{fcbphi} can be reduced to the scalar eigenvalue problem
\begin{equation*}
{\mathscr F}_{c}^{(\af)}\phi(\bx;c)=\lambda^{(\af)}\phi(\bx;c).
\end{equation*}
which link up the definition \ref{BPSWFs} of scalar ball PSWFs.
\subsection{Divergence free ball PSWFs}

Based on the previous discussions in this section, we give the following definition.
\vskip 3pt
\begin{defn} {\bf (Divergence free ball PSWFs).} For real $\alpha>-1$ and real $c\ge 0,$
the divergence free ball  prolate spheroidal wave functions
on a unit  ball $\ball,$ denoted by $\big\{\swf_{k,\ell}(\bx; c)\big\}_{(\ell,n)\in \Upsilon}^{k\in {\mathbb N}_0}$  are  eigenfunctions of the integral operator ${\mathscr F}_c^{(\alpha)}$ defined in \eqref{eigen1}, that is,
\begin{equation}\label{eigen2}
{\mathscr F}_c^{(\alpha)}[\swf_{k,{\ell}}](\bx;c)={(-\ri)}^{n+2k}\lambda_{n,k}^{(\af)}(c)\, \swf_{k,{\ell}}(\bx;c),\quad
\bx\in \ball,
\end{equation}
where $\big\{\lambda_{n,k}^{(\af)}(c)\big\}_{k,n\in \NN}$  are modulus of  the corresponding eigenvalues
defined in \eqref{eigen1}, and $c$ is the bandwidth parameter.
\end{defn}
Two main issues of the divergence free vectorial ball PSWFs on ${\mathbb B}$  need to be addressed:
\begin{itemize}
\item In the spherical-polar coordinates, divergence
free vectorial ball PSWFs $\swf_{k,\ell}(\bx;c)$ have the simple representation by the scalar ball PSWFs $\psi_{k,\ell}^{\af,n}(\bx;c)$, i.e.,
\begin{equation}\label{VPSWF}
\begin{split}
\swf_{k,\ell}(\bx;c)&=(\bx\times\nabla)\psi_{k,\ell}^{\af,n}(\bx;c),\quad  (\ell,n)\in \Upsilon,\;\; k\in\NN_0.
\end{split}
\end{equation}
\item With the aid of  Lemma \ref{lamth} and \eqref{eigen2}, we can derive that
$\big\{\swf_{k,{\ell}}(\bx;c)\big\}_{(\ell,n)\in \Upsilon}^{k\in \NN_0}$ are also the eigenfunctions
of ${\mathcal Q}_{c}^{(\af)}:$
\begin{equation*}
{\mathcal Q}_{c}^{(\af)}[\swf_{k,{\ell}}](\bx;c)=\mu_{n,k}^{(\af)}(c)\,\swf_{k,{\ell}}(\bx;c),
\end{equation*}
and the eigenvalues are the same as those defined in \eqref{lamth}  such that have the relation: $\mu_{n,k}^{(\alpha)}(c)=|\lambda_{n,k}^{(\af)}(c)|^2\,.$
\end{itemize}

It is straightforward to prove the following  orthogonality result.
\begin{thm}\label{VGpswfP}  For any $c>0$ and $\af>-1$,  divergence free vectorial ball PSWFs
 $\big\{\swf_{k,\ell}(\bx; c)\big\}_{(\ell,n)\in \Upsilon}^{k\in \NN_0}$ are all real, smooth, and  orthogonal  in
$L^2_{\varpi_{\af}}(\ball)^3$,
\begin{equation*}\label{orthnswf}
\int_{\ball} \swf_{k,\ell}(\bx; c)\cdot {\bs \psi}^{\af,m}_{j,\iota}(\bx; c)\varpi_{\af} (\bx)\d\bx=n(n+1)\delta_{k,j}\delta_{\ell ,\iota}\delta_{n,m}\,
\quad (\ell,n),(\iota,m)\in \Upsilon,\, k,j\in \NN_0.
\end{equation*}
\end{thm}
\begin{proof}In view of \eqref{eq:LaplaceBeltrami},
we have
\begin{eqnarray*}
\begin{aligned}
  (\bs{\psi}_{k,\ell}^{\alpha,n}, \bs{\psi}_{j,\iota}^{\alpha,m})_{\omega^{\alpha}}=&(\bs x \times\nabla  \psi^{\alpha,n}_{k,\ell}, \bs x \times\nabla \psi^{\alpha,m}_{j,\iota})_{\omega_{\alpha}}\\
=& (\hat{\bs x} \times\nabla_0  \psi^{\alpha,n}_{k,\ell}, \hat{\bs x} \times\nabla_0  \psi^{\alpha,m}_{j,\iota})_{\omega_{\alpha}}
\\
=& -( \Delta_0 \psi^{\alpha,n}_{k,\ell}, \psi^{\alpha,m}_{j,\iota})_{\omega_{\alpha}}=- (r^n \phi_{k}^{\alpha,n}(2r^2-1; c) \Delta_0Y^n_{\ell}(\hat\bx), \psi^{\alpha,m}_{j \iota})_{\varpi_{\af}}\\
=&n(n+1)(r^n \phi_{k}^{\alpha,n}(2r^2-1; c)Y^n_{\ell}(\hat\bx), \psi^{\alpha,m}_{j \iota})_{\varpi_{\af}}\\
  =&n(n+1)(\psi^{\alpha,n}_{k,\ell} , \psi^{\alpha,m}_{j \iota})_{\varpi_{\af}}=n(n+1)\delta_{mn}\delta_{\ell\iota}\delta_{jk}.
 \end{aligned}
  \end{eqnarray*}
This gives the proof.
  \end{proof}


\section{Divergence free ball PSWFs as  vectorial eigenfunctions of Sturm-Liouville operators}\label{sec4}
In this section, we show that the divergence free vectorial ball PSWFs are eigenfunctions of  two differential operators.

\subsection{Sturm-Liouville equations of the first kind}
For $\af>-1$, we define the  second-order Sturm-Liouville differential operator $\mathscr{L}_{\bx}^{(\af)}$ with a parameter $c\ge 0$,
\begin{align}
\label{opDxB}
&\mathscr{L}_{c,\bx}^{(\af)}:=\mathscr{L}_{\bx}^{(\af)}+ c^2 \|\bx\|^2=-(1-\|\bx\|^2)^{-\af}\nabla\cdot({\bs {\rm I}}-\bx\bx^{\tr})(1-\|\bx\|^2)^{\af}\nabla+ c^2 \|\bx\|^2,
\end{align}
for $\bx\in \ball$.
Throughout this paper, composite differential operators  are  understood in the convention of right associativity, for instance,
$$\nabla\cdot({\bs {\rm I}}-\bx\bx^{\tr})(1-\|\bx\|^2)^{\af}\nabla
=\nabla\cdot[({\bs {\rm I}}-\bx\bx^{\tr})(1-\|\bx\|^2)^{\af}\nabla].$$
Obviously,  $\mathscr{L}_{c,\bx}^{(\af)}$  extends the definition of $\mathscr{L}_{\bx}^{(\af)}$   such that
  $\mathscr{L}_{\bx}^{(\af)}=\mathscr{L}_{0,\bx}^{(\af)}$ and $\mathscr{L}_{c,\bx}^{(\af)}=\mathscr{L}_{\bx}^{(\af)}+c^2\|\bx\|^2$.

In the forthcoming discussion, we shall demonstrate that the second-order Sturm-Liouville differential operator $\mathscr{L}_{c, \bx}^{(\af)}$ only have one kind of divergence free vectorial eigenfunctions, which are divergence
free vectorial ball PSWFs.  For this purpose, we first prove the following result.
\begin{lemma}\label{Lxprop} It  holds that
 \begin{align}
 \label{L-curl}
&(\bx\times\nabla)\mathscr{L}_{c,\bx}^{(\af)}=\mathscr{L}_{c,\bx}^{(\af)}(\bx\times\nabla),
\\
 \label{L-curl2}
&\Curl\mathscr{L}_{c,\bx}^{(\af)}=[\mathscr{L}_{c,\bx}^{(\af+1)}+2\af+4]\Curl + 2\,c^2 \bx \times ,
\\
 \label{L-curl3}
&\mathscr{L}_{c,\bx}^{(\af)} \nabla \times \bx \times \nabla
= \nabla\times \bx \times \nabla  [\mathscr{L}_{c,\bx}^{(\af-1)}-(2\af+2)]
+2\,c^2\, \|\bx\| \nabla_0   .
\end{align}
\end{lemma}
\begin{proof}
We obtain from a technical reduction together with \eqref{Lap-Bel} that
\begin{align}
\label{Lc}
\begin{split}
\mathscr{L}_{c,\bx}^{(\af)}&= -(1-\|\bx\|^2)^{-\af} \nabla\cdot  (1-\|\bx\|^2)^{\af+1}\nabla - \Delta_0 +c^2 \|\bx \|^2\\
&=-[(1-\|\bx\|^2)\Delta-2(\af+1)(\bx \cdot\nabla)]- [\|\bx\|^2\Delta -(\bx \cdot\nabla) (\bx \cdot\nabla+1)  ] +c^2 \|\bx \|^2\\
&=-\Delta+(\bx\cdot\nabla)(\bx\cdot\nabla+2\af+3) +c^2 \|\bx \|^2.
\end{split}
\end{align}
Then
the identity
$\mathscr{L}_{c,\bx}^{(\af)} (\bx\times\nabla)= (\bx\times\nabla)\mathscr{L}_{c,\bx}^{(\af)}  $
is an immediate consequence of \eqref{commute1}- \eqref{commute3}.

It can be readily  checked that
\begin{equation*}
\partial_{x_i}(\bx\cdot\nabla)=(\bx\cdot\nabla+1)\partial_{x_i}, \qquad  \partial_{x_i}\|\bx\|^2 = \|\bx\|^2 \partial_{x_i} + 2x_i.
\end{equation*}
As a result,
\begin{equation*}
\begin{split}
\partial_{x_i}\mathscr{L}_{c,\bx}^{(\af)}&\overset{\eqref{Lc}}=\partial_{x_i}[-\Delta+(\bx\cdot\nabla)(\bx\cdot\nabla+2\af+3) + c^2\|\bx\|^2]\\
&=[-\Delta+(\bx\cdot\nabla+1)(\bx\cdot\nabla+2\af+4)  +  c^2\|\bx\|^2 ]\partial_{x_i}  +2\,c^2 x_i\\
&=[-\Delta+(\bx\cdot\nabla)(\bx\cdot\nabla+2\af+5)+ c^2\|\bx\|^2] \partial_{x_i}+(2\af+4)\partial_{x_i}+2\,c^2 x_i\\
&\overset{\eqref{Lc}}=\mathscr{L}_{c,\bx}^{(\af+1)}\partial_{x_i}+(2\af+4)\partial_{x_i} +2\,c^2 x_i,
\end{split}
\end{equation*}
which implies
\begin{align*}
\Curl\mathscr{L}_{c,\bx}^{(\af)}=[\mathscr{L}_{c,\bx}^{(\af+1)}+2\af+4]\Curl + 2\,c^2 \bx \times.
\end{align*}
In the sequel,
\begin{align*}
\mathscr{L}_{c,\bx}^{(\af)} \nabla \times \bx \times \nabla
&\overset{\eqref{L-curl2}}=[\nabla\times \mathscr{L}_{c,\bx}^{(\af-1)} -(2\af+2) \nabla\times -2\,c^2\bx\times]\bx \times \nabla
\\
&\overset{\eqref{L-curl}}=\nabla\times \bx \times \nabla  [\mathscr{L}_{c,\bx}^{(\af-1)}-(2\af+2)]
-2\,c^2\, \bx \times \bx \times \nabla ,
\end{align*}
which together with \eqref{opert-1g} gives the desired  result \eqref{L-curl3}.
The proof is completed.
\end{proof}

The following Lemma recall  the scalar ball PSWFs as eigenfunctions of the second-order differential
operator differential
operator $\mathscr{L}_{c,\bx}^{(\af)}$.

\begin{lemma}{\cite[Theorem 4.1]{ZhangLi.18}}
\label{ThZhLi}
 For real $\alpha>-1$ and real $c\ge 0,$
the ball prolate spheroidal wave functions, denoted by $\big\{\psi^{\alpha,n}_{k,\ell}(\bx; c)\big\}_{(\ell,n)\in \Upsilon_{\!0}}^{k\in\NN_0},$ are  eigenfunctions of the differential operator defined in  $\mathscr{L}_{c,\bx}^{(\af)}$ defined in \eqref{opDxB}, that is,
\begin{equation}\label{varphichi}
\mathscr{L}_{c,\bx}^{(\af)}\psi^{\alpha,n}_{k,\ell}(\bx; c)=\chi_{n,k}^{(\af)}(c)\, \psi^{\alpha,n}_{k,\ell}(\bx; c),\quad  \bx\in \ball,
\end{equation}
where  $c$ is the bandwidth parameter, and $\big\{\chi_{n,k}^{(\af)}(c)\big\}_{k,n\in \NN_0}$ are  real, positive eigenvalues ordered  for fixed $n$ as follows
\begin{equation}\label{eqn_incr}
0<\chi_{n,0}^{(\af)}(c)<\chi_{n,1}^{(\af)}(c)<\cdots<
\chi_{n,k}^{(\af)}(c)<\cdots.
\end{equation}

\end{lemma}
With the aid of Lemma \ref{Lxprop}, Lemma \ref{BPSWFs} and the relation \eqref{VPSWF}, we can derive the following
result,
\begin{align}
\begin{split}
\mathscr{L}_{c,\bx}^{(\af)}{\bs\psi}^{n}_{k,\ell}(\bx;c)&=\mathscr{L}_{c,\bx}^{(\af)} \bx \times \nabla \psi^{n}_{k,\ell} (\bx;c)
\overset{\eqref{L-curl}}= \bx \times \nabla  \mathscr{L}_{c,\bx}^{(\af)}  \psi^{n}_{k,\ell} (\bx;c)
\\
&\overset{\eqref{varphichi}}=  \chi^{(\alpha)}_{n,k}(c)  \bx \times \nabla\psi^{n}_{k,\ell} (\bx;c) =  \chi^{(\alpha)}_{n,k}(c)  {\bs \psi}^{n}_{k,\ell} (\bx;c),
\end{split}
\label{opD}
\end{align}
which means that  any $(\chi^{(\alpha)}_{n,k},  {\bs \psi}^{n}_{k,\ell} )$ with $(\ell,n) \in \Upsilon, \, k\in \NN_0$  is  a divergence free vectorial pair of
$\mathscr{L}_{c,\bx}^{(\af)}$.
\begin{rem}We would like to point out that Lemma \ref{Lxprop} provides us an approach for solving the vectorial eigen problem \eqref{opD}  with the divergence free constraint
via the  scalar eigen problem. More precisely, \eqref{L-curl} states that,
$(\chi,\bx \times \nabla \phi)$ is a divergence free vectorial eigen pair of  $\mathscr{L}_{c,\bx}^{(\af)}$
if  $(\chi,\phi)$ is an eigen pair of $\mathscr{L}_{c,\bx}^{(\af)}$ with  $\bx \times \nabla \phi\ne 0$; conversely,  $(\chi,\phi)$ is an eigen pair of $\mathscr{L}_{c,\bx}^{(\af)}$ if $(\chi,\bx \times \nabla \phi)$ is a
vectorial eigen pair of  $\mathscr{L}_{c,\bx}^{(\af)}$.
\end{rem}

Next, we shall demonstrate that  the Strum-Liouville operator  $\mathscr{L}_{c,\bx}^{(\af)}$ can not have other
divergence free vectorial eigenfunctions in $ L^2_{\varpi_{\af}}(\ball)^3$. Thanks to the form of \eqref{DFconst},  it suffices to show that
any   ${\bs E}(\bx) =\nabla \times \bx \times \nabla E(\bx)   $ can not be
a divergence free vectorial eigenfunctions of $\mathscr{L}_{c,\bx}^{(\af)} $.
Indeed,
\begin{align}
\begin{split}
\mathscr{L}_{c,\bx}^{(\af)}{\bs E} (\bx)&=\mathscr{L}_{c,\bx}^{(\af)} \nabla \times \bx \times \nabla  E (\bx)
\\
&\overset{\eqref{L-curl3}}
=\nabla\times \bx \times \nabla  [\mathscr{L}_{c,\bx}^{(\af-1)}-(2\af+2)] E(\bx)
+2\,c^2\, r \nabla_0  E(\bx).
\end{split}
\label{opD2}
\end{align}
Thus by \eqref{opert-1h},  \eqref{opert-1i}, and \eqref{eq:LaplaceBeltrami},  the following equality holds
\begin{align*}
0=\nabla \cdot \mathscr{L}_{c,\bx}^{(\af)}{{\bs E}(\bx)}(\bx)
= 2\,c^2\,\Delta_0 E (\bx) ,
\end{align*}
if and only if $E (\bx) = \theta(r) Y^0_1(\hat {\bx})$.  In return,  the spherical component of  $E$ is constant and
$ \nabla \times\bx \times \nabla E (\bx) = \nabla \times \hat \bx \times[ \nabla_0 E (\bx)] =0$. This fact indicates that the  vectorial functions $\nabla \times \hat \bx \times  \nabla E (\bx) $
are not eigenfunctions of the Strum-Liouville operator  $\mathscr{L}_{c,\bx}^{(\af)}$.

Finally, a combination of the previous discussions leads to the following result.
\begin{thm}\label{VBPSWFs} For real $\alpha>-1$ and real $c\ge 0,$
the divergence free ball PSWFs $\swf_{k,\ell}(\bx; c),\, (\ell,n)\in \Upsilon,\,  k\in \NN_0$, are  eigenfunctions of the differential operator $\mathscr{L}_{c,\bx}^{(\af)}$,
\begin{equation}\label{varphichi2}
\mathscr{L}_{c,\bx}^{(\af)}\swf_{k,\ell}(\bx; c)=\chi_{n,k}^{(\af)}(c)\, \swf_{k,\ell}(\bx; c),\quad  \bx\in \ball,
\end{equation}
where the eigenvalues $\big\{\chi_{n,k}^{(\af)}(c)\big\}_{k,n\in {\mathbb N}}$ are defined as in Lemma \ref{ThZhLi}.
\end{thm}
\begin{rem}If $c=0,$ we find readily from the previous discussions that
\begin{equation}\label{gtog}
\swf_{k,\ell}(\bx; 0)=(\bx\times\nabla)P_{k,\ell}^{\af,n}(\bx),\quad \chi_{n,k}^{(\af)}(0)=\gamma_{n+2k}^{(\af)}.
\end{equation}
\end{rem}


\subsection{Sturm-Liouville equations of the second kind}

For $\af>-1$, we define the   second-order Sturm-Liouville operators of the second kind,
\begin{align}
&\mathscr{D}_{c,\bx}^{(\af)}:=(1-\|\bx\|^2)^{-\af}\nabla\times(1-\|\bx\|^2)^{\af+1}\nabla\times -\Delta_0 + c^2 \|\bx\|^2,
\end{align}
for $\bx\in \ball,$ and real $c\ge 0.$ We are able to show that the divergence free ball PSWFs $\swf_{k,\ell}(\bx; c)$ are the eigenfunctions of the  second-order Sturm-Liouville operators $\mathscr{D}_{c,\bx}^{(\af)}$.  As a preparation, we first obtain the following result.
\begin{lemma}\label{Dxprop} It  holds that
 \begin{align}
 \label{D-curl}
&(\bx\times\nabla)\mathscr{D}_{c,\bx}^{(\af)}=(\bx \times \nabla ) [\mathscr{L}_{c,\bx}^{(\af)} +2(\alpha+1)],
\\
 \label{D-curl2}
&\mathscr{D}_{c,\bx}^{(\af)} \nabla\times \bx \times \nabla   =  \nabla\times  \bx \times \nabla \mathscr{L}_{c,\bx}^{(\af-1)}
 +2c^2 \|\bx \| \nabla_0  -2(\alpha+1) \nabla  \Delta_0.
  \end{align}
\end{lemma}
\begin{proof} Firstly,  it can be readily shown  that
\begin{align}
\label{CurlCurl}
&(\Curl)^2 = \nabla  \nabla \cdot  - \Delta,
\\
\label{xxCurl}
& \bx \times \Curl  = \nabla \bx \cdot - (\bx \cdot \nabla +1).
\end{align}
From \eqref{Lap-Bel},  \eqref{CurlCurl}, \eqref{xxCurl} and \eqref{Lc}, we observe that
\begin{align}
\begin{split}
\mathscr{D}_{c,\bx}^{(\af)}=&(1-|\bs x|^2)^{-\alpha}  \Curl (1-|\bs x|^2)^{\alpha+1}  \Curl  -\Delta_0 + c^2 \|\bx \|^2
 \\
 = & \big[ (1-|\bs x|^2)^{-\alpha}  [ (1-|\bs x|^2)^{\alpha+1} \Curl -  2(\alpha+1) (1-|\bs x|^2)^{\alpha}  \bx \times ]  \Curl
\\
  &-[ \|\bx\|^2 \Delta - \bx \cdot \nabla (\bx \cdot \nabla+1) ]+ c^2 \|\bx \|^2
 \\
 =&  (1-|\bs x|^2) [ \nabla  \nabla \cdot  - \Delta ]  -  2(\alpha+1) [ \nabla \bx \cdot - (\bx \cdot \nabla +1) ]
 \\
& - \|\bx\|^2 \Delta + \bx \cdot \nabla (\bx \cdot \nabla+1) + c^2 \|\bx \|^2
 \\
 =&  \big[( \bs x \cdot \nabla +2\alpha+2) (\bs x \cdot \nabla+1)  - \Delta + c^2 \|\bx \|^2  \big]
+ \big[(1-|\bs x|^2) \nabla  \nabla \cdot   -2(\alpha+1) \nabla \bx \cdot  \big]
\\
=& \mathscr{L}_{c,\bx}^{(\af)} +2(\alpha+1)  + \big[(1-|\bs x|^2) \nabla  \nabla \cdot   -2(\alpha+1) \nabla \bx \cdot  \big].
\end{split}
\label{Dc}
\end{align}
Then, we have
\begin{align*}
  \mathscr{D}_{c,\bx}^{(\af)}  (\bx \times \nabla )
\overset{\eqref{Dc}}=\,& [\mathscr{L}_{c,\bx}^{(\af)} +2(\alpha+1)]  (\bx \times \nabla )
+ \big[(1-|\bs x|^2) \nabla  \nabla \cdot   -2(\alpha+1) \nabla \bx \cdot  \big]  (\bx \times \nabla )
  \\
\overset{\eqref{opert-1d}}=& (\bx \times \nabla ) [\mathscr{L}_{c,\bx}^{(\af)} +2(\alpha+1)],
\end{align*}
which yields the identity  \eqref{D-curl}.

Now, it remains to establish \eqref{D-curl2}.  Since $\nabla \cdot \nabla \times =0$,
we find from \eqref{Dc}, \eqref{L-curl3} and \eqref{opert-1c} that
\begin{align*}
\mathscr{D}_{c,\bx}^{(\af)} \nabla\times \bx \times \nabla  = &  [\mathscr{L}_{c,\bx}^{(\af)}  +2(\alpha+1)] \nabla\times   \bx \times \nabla
-2(\alpha+1) \nabla \bx \cdot   \nabla\times   \bx \times \nabla
\\
 = & \nabla\times  \bx \times \nabla \mathscr{L}_{c,\bx}^{(\af-1)}   +2c^2 \|\bx \| \nabla_0
  -2(\alpha+1) \nabla  \Delta_0,
\end{align*}
which gives the desired result \eqref{D-curl2}.
  \end{proof}

Obviously, using the formula \eqref{D-curl},  we
deduce from \eqref{varphichi} that
\begin{align*}
\begin{split}
\mathscr{D}_{c,\bx}^{(\af)}{\bs\psi}^{n}_{k,\ell}(\bx;c)&=\mathscr{D}_{c,\bx}^{(\af)} \bx \times \nabla \psi^{n}_{k,\ell} (\bx;c)
\overset{\eqref{D-curl}}=  (\bx \times \nabla ) [\mathscr{L}_{c,\bx}^{(\af)} +2(\alpha+1)]  \psi^{n}_{k,\ell} (\bx;c)
\\
&\overset{\eqref{varphichi}}= [\chi^{(\alpha)}_{n,k}(c)+2\af+2] \bx \times \nabla\psi^{n}_{k,\ell} (\bx;c) = [\chi^{(\alpha)}_{n,k}(c)+2\af+2]  {\bs \psi}^{n}_{k,\ell} (\bx;c).
\end{split}
\end{align*}

\begin{rem}The above argument states that,
$(\chi+2\alpha+2,\bx \times \nabla \phi)$ is a divergence free vectorial eigen pair of  $\mathscr{D}_{c,\bx}^{(\af)}$
if  $(\chi,\phi)$ is an eigen pair of $\mathscr{L}_{c,\bx}^{(\af)}$ with  $\bx \times \nabla \phi\ne 0$; conversely,  $(\chi,\phi)$ is an eigen pair of $\mathscr{L}_{c,\bx}^{(\af)}$ if $(\chi+2\alpha+2,\bx \times \nabla \phi)$ is a
vectorial eigen pair of  $\mathscr{D}_{c,\bx}^{(\af)}$.
\end{rem}

Next,  we  intend to show that ${\mathscr D}_{c,\bx}^{(\alpha)}$ does not possess any
 divergence free vectorial  eigenfunction of form
\begin{equation*}
 {\bs E} (\bx) =\nabla \times (\bx \times \nabla ) E(\bx),
 \quad  E(\bx) = \theta(r) Y_{\ell}^{n}(\hat \bx), \qquad n\in \NN.
\end{equation*}
Otherwise, we observe
from \eqref{D-curl2} that
\begin{align*}
&\nabla \cdot \mathscr{D}_{c,\bx}^{(\af)} \nabla\times (\bx \times \nabla)
 \overset{\eqref{opert-1i}}=2c^2 \nabla \cdot ( \|\bx \| \nabla_0)  -2(\alpha+1) \nabla \cdot  \nabla \Delta_0
 \\
 \overset{\eqref{opert-1b}}=\,& 2c^2 (\frac1r\nabla_0 +\hat \bx \partial_r) \cdot ( r \nabla_0)  -2(\alpha+1) \Delta \Delta_0
 \\
 \overset{\eqref{opert-1a}}=\,& 2\big[c^2 - (\alpha+1) \Delta] \Delta_0
  \overset{\eqref{Lap-Bel}}= 2\Big[c^2 - (\alpha+1) \big(\partial_r^2 + \frac2r\partial_r + \frac{\Delta_0}{r^2}\big)\Big] \Delta_0,
\end{align*}
and
\begin{align*}
\begin{split}
0=\,&\chi  \nabla \cdot  {\bs E} (\bx)  =\nabla \cdot \mathscr{D}_{c,\bx}^{(\af)} {\bs E} (\bx) = \nabla \cdot \mathscr{D}_{c,\bx}^{(\af)}  \nabla\times \bx \times \nabla [ \theta(r) Y_{\ell}^{n}(\hat \bx)]
\\
 \overset{\eqref{eq:LaplaceBeltrami}}
 =\,& -2 n(n+1)  \big[c^2    - (\alpha+1) (  \partial_r^2+\frac{2}{r}\partial_r -\frac{n(n+1)}{r^2} )  \big] \theta(r)  Y^n_{\ell}(\hat {\bx}).
 \end{split}
\end{align*}
Equivalently,
\begin{align}
\begin{split}
\big[c^2    - (\alpha+1) (  \partial_r^2+\frac{2}{r}\partial_r -\frac{n(n+1)}{r^2} )  \big] \theta(r) =0
 \end{split}
 \label{DivD}
\end{align}
Meanwhile,  we deuce that
\begin{align*}
&  \chi   \bx \cdot \nabla\times \bx \times \nabla E
=\bx \cdot \mathscr{D}_{c,\bx}^{(\af)} \nabla\times \bx \times \nabla E
\\
\overset{ \eqref{D-curl2}}=\,& \bx \cdot \big[ \nabla\times  \bx \times \nabla \mathscr{L}_{c,\bx}^{(\af-1)}
 +2c^2 \cdot  \|\bx \| \nabla_0
   -2(\alpha+1) \cdot \nabla  \Delta_0\big]E
 \\
\overset{\eqref{opert-1c}}{\underset{\eqref{opert-1a}}=}\,& \Delta_0 [\mathscr{L}_{c,\bx}^{(\af-1)}  - 2(\alpha+1)r\partial_r]E
 \\
 \overset{\eqref{Ldef}}=\,& \Delta_0[(r^2-1)\partial^2r-\frac{2}{r}\partial r-\frac{1}{r^2} \Delta_0]E
 \\
 =\,& \Delta_0[(r^2-1)  (  \partial_r^2+\frac{2}{r}\partial_r +\frac{\Delta_0}{r^2} )  - 2r\partial_r - \Delta_0 ]E,
 \end{align*}
and,
\begin{align*}
-n(n+1)\chi \,&   \theta(r)Y^n_{\ell}(\hat {\bx})
\overset{\eqref{eq:LaplaceBeltrami}}{\underset{\eqref{opert-1c}}=}   \chi   \bx \cdot \nabla\times \bx \times \nabla { E}({\bx})
=\bx \cdot \mathscr{D}_{c,\bx}^{(\af)} \nabla\times \bx \times \nabla   { E}( {\bx})
\\
 \overset{\eqref{eq:LaplaceBeltrami}}=\,& -n(n+1) \big[ (r^2-1) \big( \partial_r^2+\frac{2}{r}\partial_r -\frac{n(n+1)}{r^2} \big)  - 2r\partial_r +n(n+1) \big] \theta(r)Y^n_{\ell}(\hat {\bx})
\\
 \overset{\eqref{DivD}}=\,& -n(n+1) \big[ (r^2-1) \frac{c^2}{\alpha+1} - 2r\partial_r +n(n+1) \big] \theta(r)Y^n_{\ell}(\hat {\bx}) .
\end{align*}
As a result,
\begin{align}
\label{DivX}
\big[(r^2-1) \frac{c^2}{\alpha+1} - 2r\partial_r +n(n+1)\big] \theta(r) =  \chi \theta(r).
\end{align}
This, in return, gives
  \begin{align*}
 0\overset{\eqref{DivD}}=\,&  \Big[r^2\partial_r^2 +2r\partial_r-n(n+1) -\frac{c^2r^2}{\alpha+1} \Big]\theta(r)
 =\Big[ r\partial_r (r\partial_r+ 1) -n(n+1) -\frac{c^2r^2}{\alpha+1} \Big]\theta(r)
 \\
 \overset{\eqref{DivX}}=\,& \frac{r}2 \partial_r \Big[   \frac{c^2(r^2-1)}{\alpha+1} +n(n+1) -\chi  + 2\Big] \theta(r)
 -\Big[n(n+1) +\frac{c^2r^2}{\alpha+1} \Big]\theta(r)
  \\
  =\,& \frac{r}2 \Big[   \frac{c^2(r^2-1)}{\alpha+1} +n(n+1) -\chi  + 2\Big] \partial_r  \theta(r)
  -n(n+1)   \theta(r)
  \\
  \overset{\eqref{DivD}}=\,& \frac{1}4 \Big[   \frac{c^2(r^2-1)}{\alpha+1} +n(n+1) -\chi  + 2\Big]  \Big[   \frac{c^2(r^2-1)}{\alpha+1} +n(n+1) -\chi \Big] \theta(r)
  -n(n+1)   \theta(r)
  \\
  =\,& \frac{1}4 \Big[   \frac{c^2(r^2-1)}{\alpha+1} +n(n-1)-\chi \Big]  \Big[   \frac{c^2(r^2-1)}{\alpha+1} +(n+2)(n+1) -\chi \Big]  \theta(r),
  \end{align*}
  which  states that  $ \theta(r)=0$  for $\alpha>-1$ and $c>0$, thus $\bs E$ is not an eigenfunction.

With the aid of the above discussion, we can obtain the following  result.
\begin{thm}\label{VBPSWFs2} For real $\alpha>-1$ and real $c\ge 0,$
the divergence free ball PSWFs $\swf_{k,\ell}(\bx; c),\, (\ell,n)\in \Upsilon,\,  k\in \NN_0$, are  eigenfunctions of the differential operator $\mathscr{D}_{c,\bx}^{(\af)}$,
\begin{equation}\label{varphichi2}
\mathscr{D}_{c,\bx}^{(\af)}\swf_{k,\ell}(\bx; c)=[\chi_{n,k}^{(\af)}(c) + 2\alpha+2]\, \swf_{k,\ell}(\bx; c),\quad  \bx\in \ball.
\end{equation}
\end{thm}
\section{Numerical evaluation of the  divergence free ball PSWFs}\label{sec5}

In this section, we present an efficient algorithm to evaluate the divergence free vectorial ball PSWFs and their associated eigenvalues.
\subsection{Spectrally accurate Bouwkamp algorithm}

\begin{defn}
With the polynomials $P_{k,\ell}^{\af,n}(\bx)$ on a ball $\ball$ in $\mathbb{R}^3$, we define  one kind of vectorial ball polynomials $\P_{k,\ell}^{\af,n}(\bx)$,
\begin{equation*}
{\P}_{k,\ell}^{\af,n}(\bx)=(\bx \times \nabla)P_{k,\ell}^{\af,n}(\bx),\quad \bx\in \ball, \;\;  (\ell,n) \in \Upsilon,\;\; k\in \NN_0.
\end{equation*}
\end{defn}
Next, we introduce below some basic properties of  the vector ball polynomials $\bs{P}_{k,\ell}^{\alpha,n}(\bs x)$, which be used for the computation of divergence free vectorial ball PSWFs.
\begin{lemma}
 $\bs{P}_{k,\ell}^{\alpha,n}(\bs x)$ satisfies the divergence free constraint, i.e.,
\begin{align}\label{VBP-divfree}
   \nabla\cdot \bs{P}_{k,\ell}^{\alpha,n}(\bs x) =0,\;\;\bx\in \ball,\;\;  (\ell,n) \in \Upsilon,\;\; k\in \NN_0.
    \end{align}
 \end{lemma}
\begin{proof}
A direct calculation leads to
\begin{align*}
 \nabla\cdot\P^{\alpha,n}_{k,\ell}({\bs x})= \nabla\cdot(\bs x \times\nabla P^{\alpha,n}_{k,\ell}({\bs x}))=-\nabla\cdot(\nabla\times\bs x P^{\alpha,n}_{k,\ell}({\bs x}))=0.
  \end{align*}
The proof is now completed.
\end{proof}

\begin{thm}\label{Th-slp}
The corresponding eigenvalue equation based on the $curl$-operator for $\bs{P}_{k,\ell}^{\alpha,n}(\bs x)$ reads
\begin{align}\label{FVBP-sl}
   \begin{aligned}
\mathscr{L}_{\bx}^{(\af)}{\P}_{k,\ell}^{\alpha,n}(\bs x):&= \big(-(1-|\bs x|^2)^{-\alpha}  \nabla \cdot
 (1-|\bs x|^2)^{\alpha+1}  \nabla
 -\ \Delta_0\big)\P^{\alpha,n}_{k,\ell}(\bs x)
 \\
&
 =(n+2k)(n+2k+2\alpha+3) \P^{\alpha,n}_{k,\ell}(\bs x).
  \end{aligned}
  \\
   \begin{aligned}\label{DxP}
\mathscr{D}_{\bx}^{(\af)}{\P}_{k,\ell}^{\alpha,n}(\bs x):&= \big((1-|\bs x|^2)^{-\alpha}  \nabla\times
 (1-|\bs x|^2)^{\alpha+1}  \nabla\times
 -\ \Delta_0\big)\P^{\alpha,n}_{k,\ell}(\bs x)
 \\
&
 =(n+2k+1)(n+2k+2\alpha+2) \P^{\alpha,n}_{k,\ell}(\bs x).
  \end{aligned}
  \end{align}
 \end{thm}
\begin{proof} Using Lemma \ref{dxbop} and Lemma \ref{Lxprop}, we obtain
\begin{equation*}
\begin{split}
\mathscr{L}_{\bx}^{(\af)}\P_{k,\ell}^{\af,n}(\bs x)&=\mathscr{L}_{0,\bx}^{(\af)}(\bx \times \nabla) P_{k,\ell}^{\af,n}(\bs x)=(\bx \times \nabla)\mathscr{L}_{0,\bx}^{(\af)} P_{k,\ell}^{\af,n}(\bs x),\\
&=(\bx \times \nabla)(n+2k)(n+2k+2\af+3)P_{k,\ell}^{\af,n}(\bs x).
\end{split}
\end{equation*}
We now prove \eqref{DxP}. From Lemma \ref{dxbop}, Lemma \ref{Lxprop} and  Lemma \ref{Dxprop}, we have
\begin{equation*}
\begin{split}
\mathscr{D}_{\bx}^{(\af)}{\P}_{k,\ell}^{\alpha,n}(\bs x)&=\mathscr{D}_{0,\bx}^{(\af)}(\bx \times \nabla)P_{k,\ell}^{\af,n}(\bx)=(\bx \times \nabla ) [\mathscr{L}_{0,\bx}^{(\af)} +2(\alpha+1)]\\
&=(\bx \times \nabla)(n+2k+1)(n+2k+2\af+2)P_{k,\ell}^{\af,n}(\bs x)
\end{split}
\end{equation*}
This ends the proof.
\end{proof}

We now use the Bouwkamp-type algorithm to evaluate $\big\{\bs\psi^{\alpha,n}_{k,\ell}, \chi_{n,k}^{(\alpha)}\big\}$ with $2k+n\le N$.
Following  the truncation rule  in \cite{Boyd.acm,wang2009new}, we set $ M=2N+2\alpha+30$ and
suppose   $\big\{\tilde {\bs\psi}^{\alpha,n}_{k,\ell}(\bx;c), \tilde \chi_{n,k}^{(\alpha)}\big\}$ to be the  approximation of $\big\{{\bs\psi}^{\alpha,n}_{k,\ell}(\bx;c), \chi_{n,k}^{(\alpha)}\big\}$ with
\begin{align*}
 \tilde{\bs\psi}^{\alpha,n}_{k,\ell}(\bs x;c)
= \sum_{j=0}^{\lceil \frac{M-n}{2} \rceil}
\tilde \beta^{n,k}_{j} \P^{\alpha,n}_{j,\ell}(\bs x),\quad  2k+n\le N.
\end{align*}
Denote  $K=\lceil \frac{M-n}{2} \rceil$. Thanks to  the Theorem \ref{VBPSWFs} and the  three-term recurrence relation \eqref{Jacobi} of the normalized Jacobi polynomials. The  coefficient $\{\tilde\beta^{n,k}_{j}\}$ can be equivalently deduced  from evaluating the radial component $\phi^{\af,n}_{k}$
of $\psi^{\alpha,n}_{k,\ell}(\bs x)= \phi^{\af,n}_{k}(2\|\bs x\|^2-1)Y^n_{\ell}(\bs x)$  in terms of
 Jacobi polynomials with the unknown coefficients $\{\tilde\beta_{j}^{n,k}\}$:
\begin{equation}\label{genexp}
\phi_{k}^{\alpha,n}(\eta; c)=\sum_{j=0}^{\infty}\tilde\beta_{j}^{n,k} J_{j}^{(\af,\beta_n)}(\eta).
\end{equation}Then the Bouwkamp-type algorithm gives  the following  finite algebraic eigen-system for $\{ \tilde \beta^{n,k}_{j}\}_{j=0}^K$ and
$\tilde \chi_{n,k}^{(\alpha)}$ (cf. \cite{ZhangLi.18}),
\begin{equation}\label{Feigsys}
({\bs A}-\tilde \chi_{n,k}^{(\alpha)}\cdot{\bs I})\vec{\beta}^{n,k}=\bs 0,
\end{equation}
where
$\vec{\beta}^{n,k}=(\tilde\beta^{n,k}_{0},\tilde\beta^{n,k}_{1},\dots,\tilde\beta^{n,k}_{K})$
and ${\bs A}$ is the $(K+1)\times (K+1)$ symmetric tridiagonal matrix whose nonzero entries are given by
\begin{equation}\label{333term}
\begin{split}
A_{j,j}=\gamma_{n+2j}^{(\af)}+\big(b_{j}^{(\af,\beta_n)}+1\big)\cdot \frac{c^2}{2};\quad  A_{j,j+1}=A_{j+1,j}=a_{j}^{(\af,\beta_n)} \cdot \frac{c^2}{2}, \quad  0\le j \le K.
\end{split}
\end{equation}

\subsection{Numerical results}
 We  plot some samples of the $\bs\psi_{k,\ell}^{\af,n}(\bs x;c)$  obtained from the previously described algorithms. Figure \ref{fg51} -\ref{fg52} visualize of
${\bs \psi}^{\af,n}_{k,\ell}(\bs x;c)$ with different $k,\ell,n,\alpha$ and $c$.
\begin{figure}\label{fg51}
  \centering
  \subfigure[$(\af, n, k, \ell)=(0,1,0,1).$]{
    \includegraphics[width=2.7in]{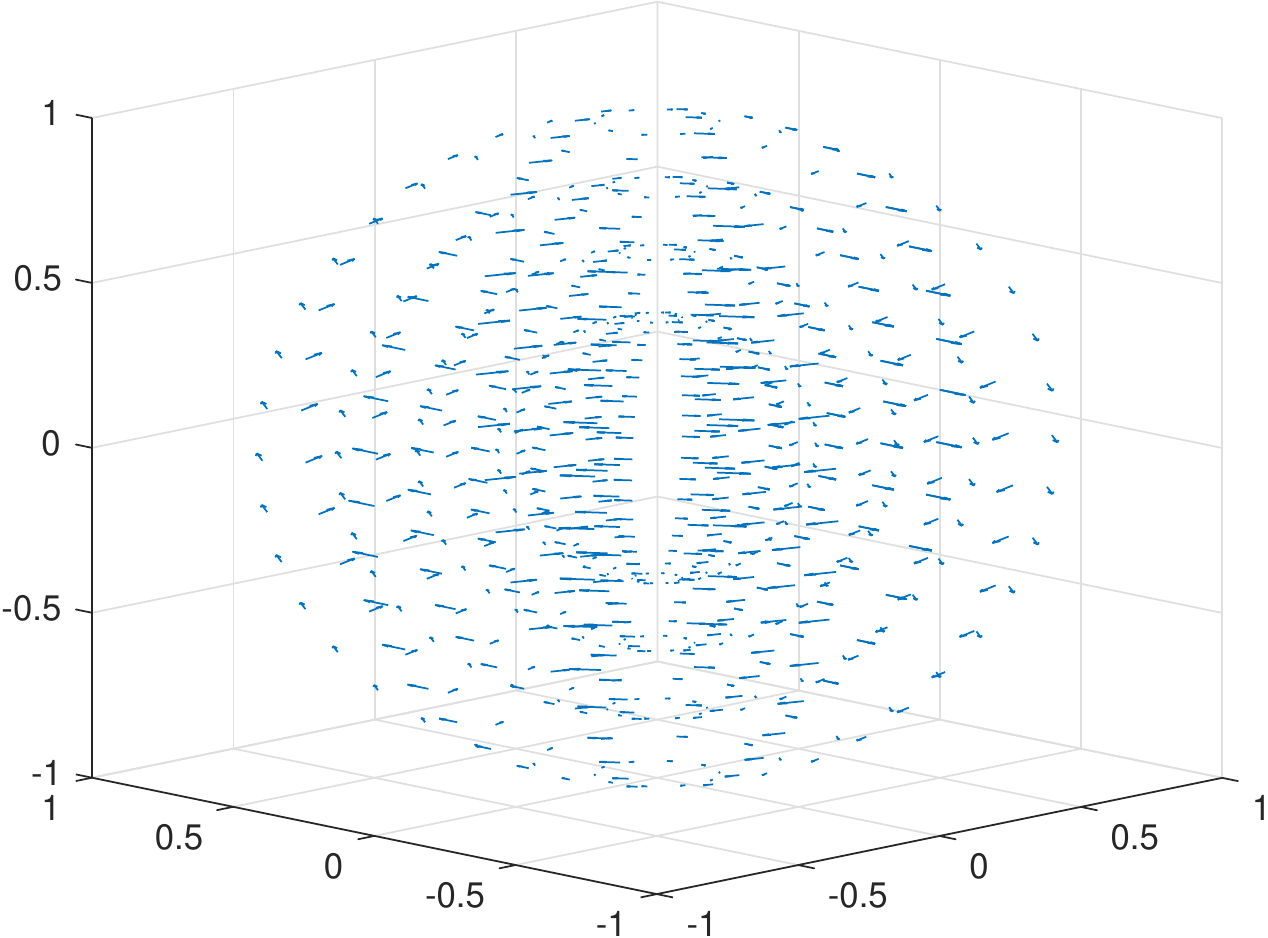}}
  \hspace{0in}
  \subfigure[$(\af, n, k,\ell)=(0,1,0,2).$]{
    \includegraphics[width=2.7in]{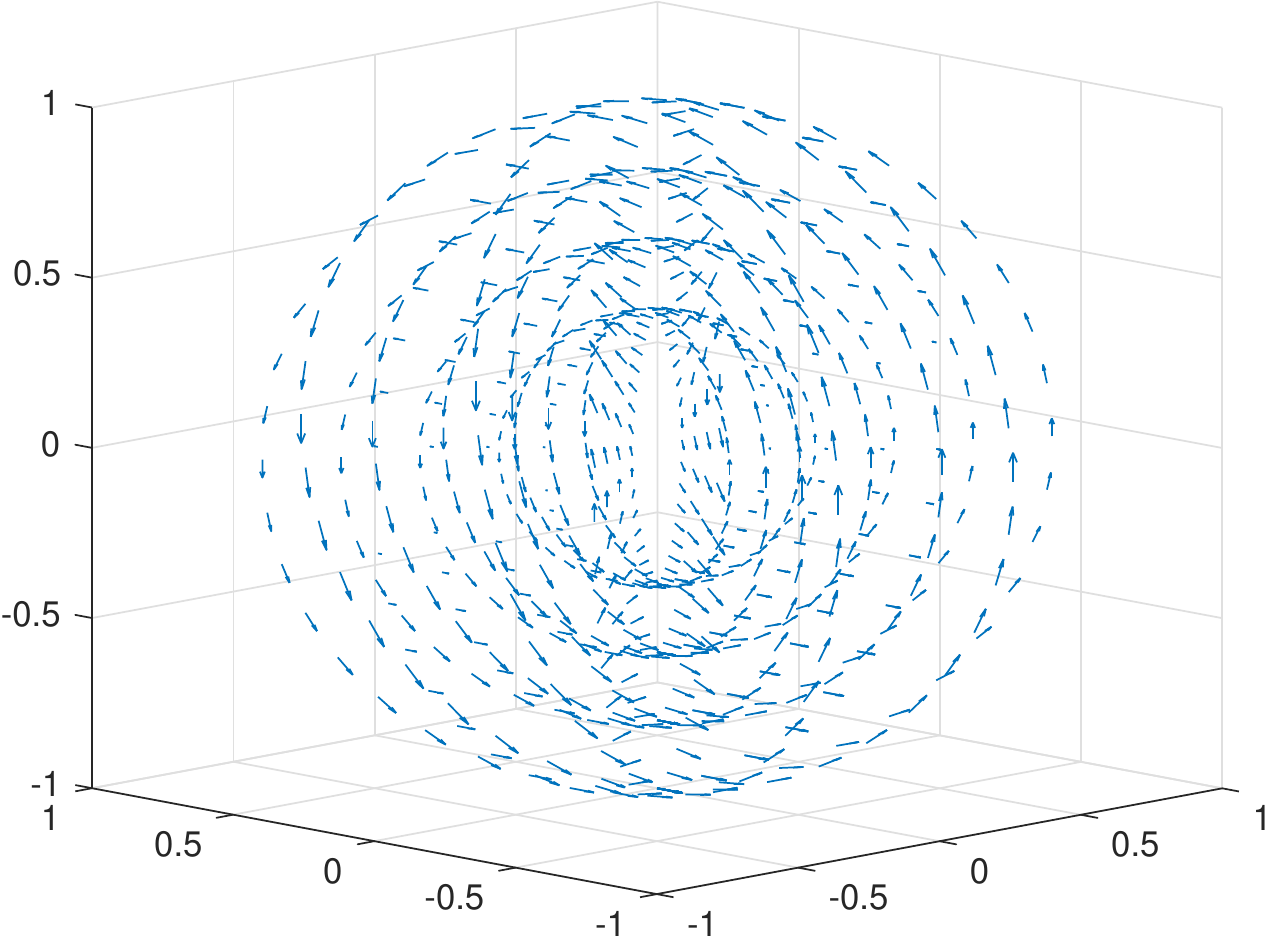}}
  \vfill
  \subfigure[$(\af, n, k,\ell)=(0,2,0,1).$]{
    \includegraphics[width=2.7in]{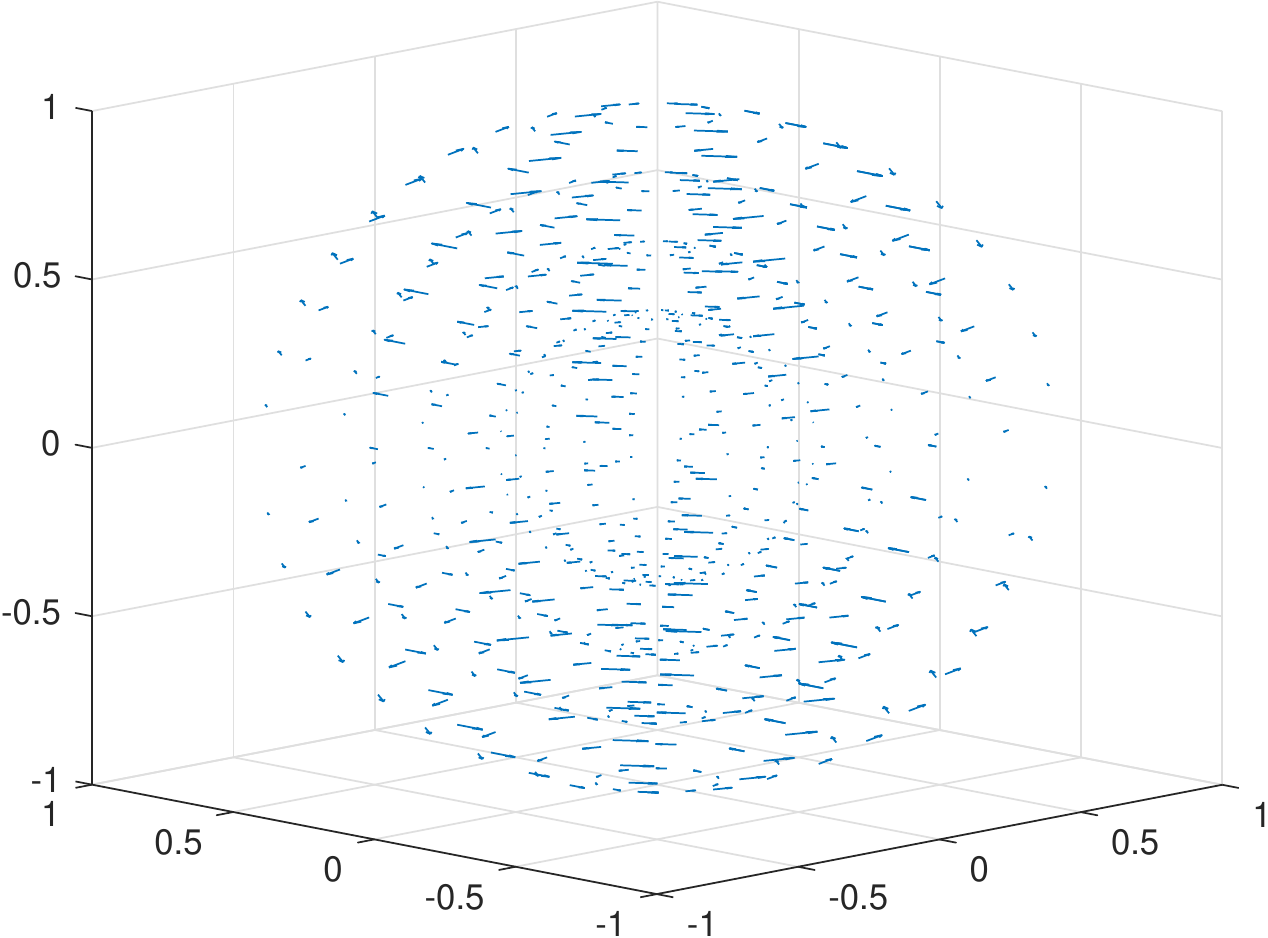}}
  \hspace{0in}
  \subfigure[$(\af, n, k, \ell)=(0,2,0,2).$]{
    \includegraphics[width=2.7in]{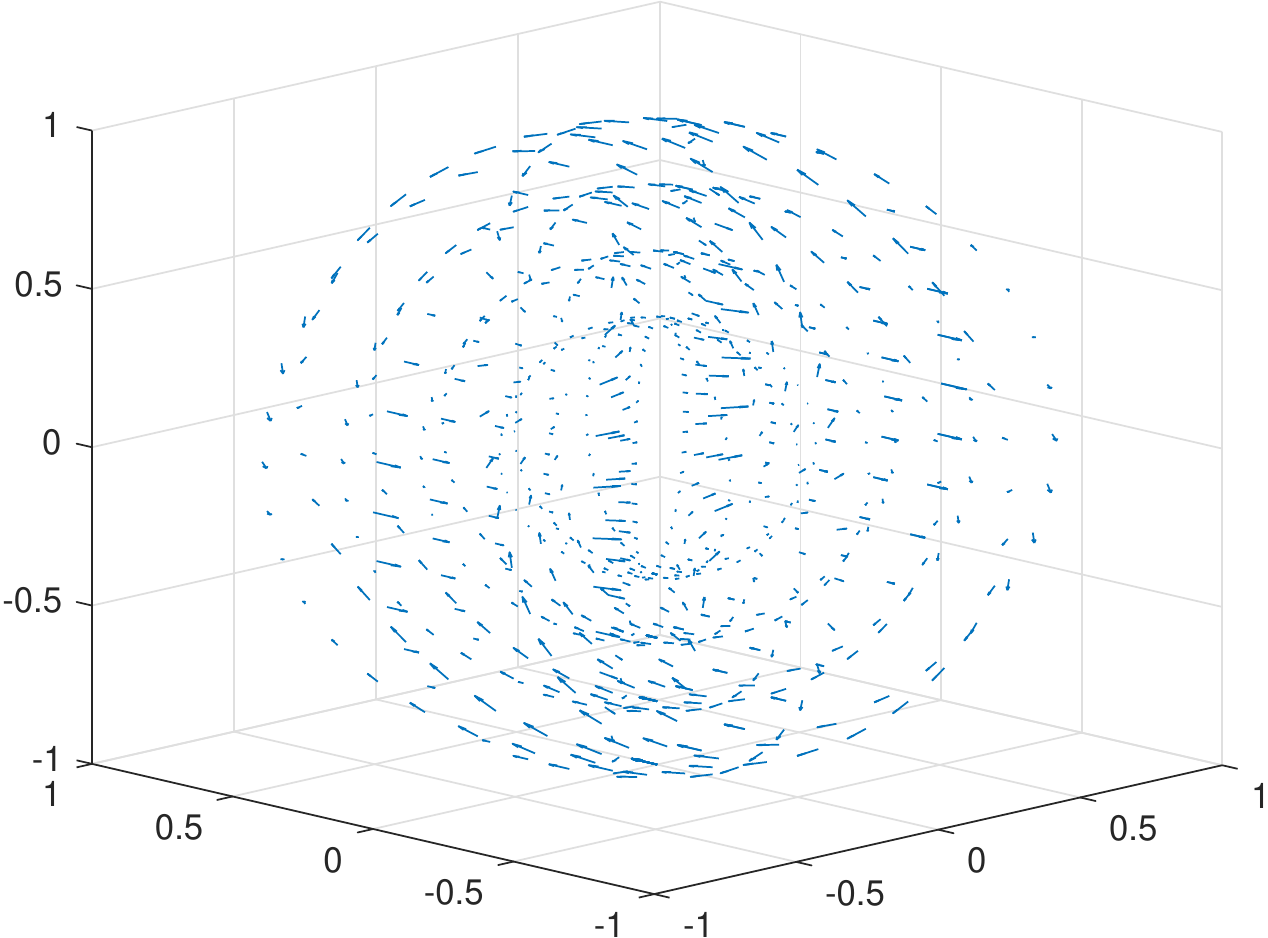}}
  \caption{$\bs\psi_{k,\ell}^{\af,n}(\bs x;c)$ with $c=2$.}
\end{figure}
\begin{figure}\label{fg52}
  \centering
  \subfigure[$(\af, n, k,\ell)=(1,1,0,2).$]{
    \includegraphics[width=2.7in]{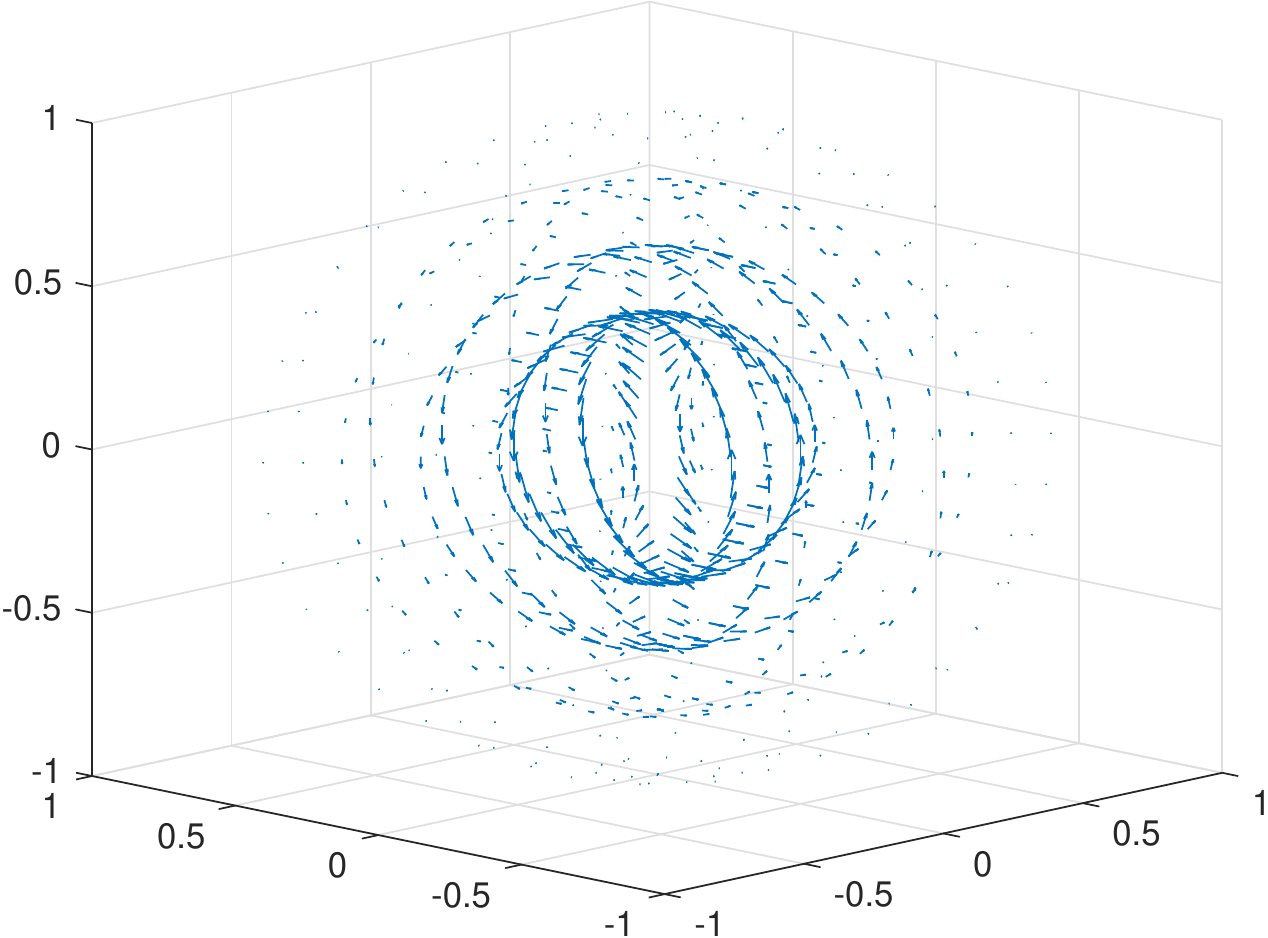}}
  \hspace{0in}
  \subfigure[$(\af, n, k, \ell)=(1,1,1,2).$]{
    \includegraphics[width=2.7in]{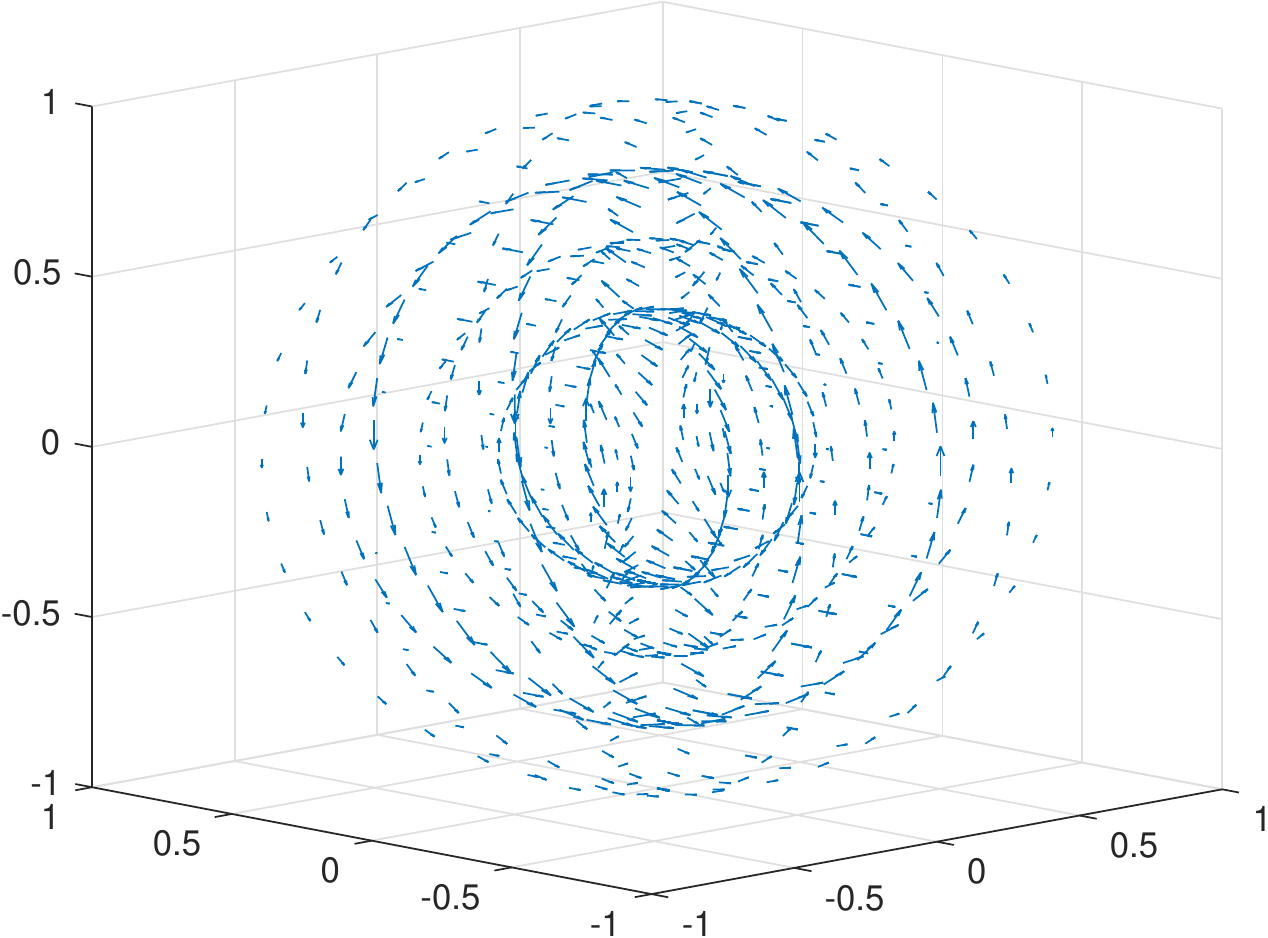}}
  \vfill
  \subfigure[$(\af, n, k, \ell)=(1,2,1,2).$]{
    \includegraphics[width=2.7in]{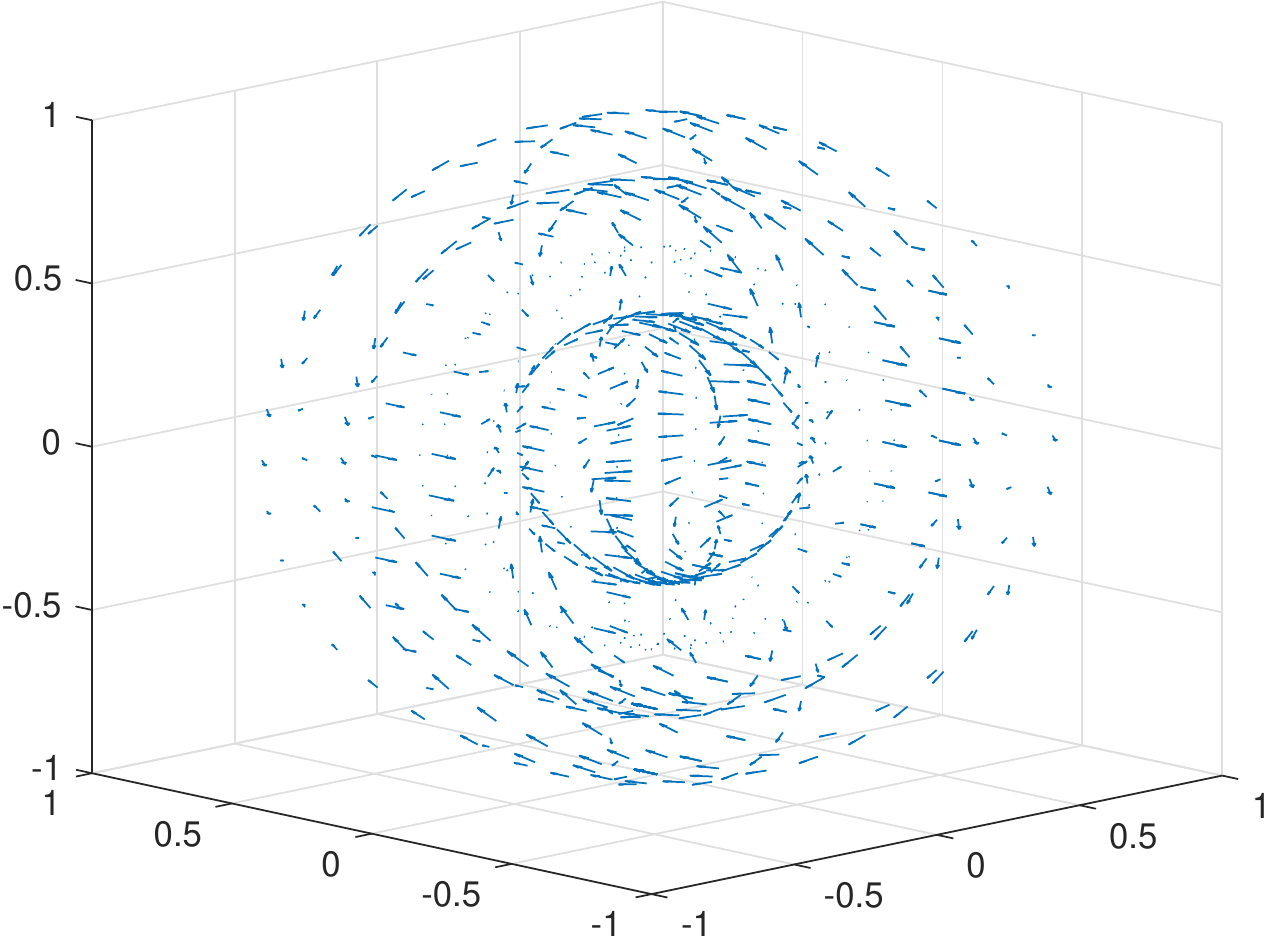}}
  \hspace{0in}
  \subfigure[$(\af, n, k, \ell)=(1,2,2,1).$]{
    \includegraphics[width=2.7in]{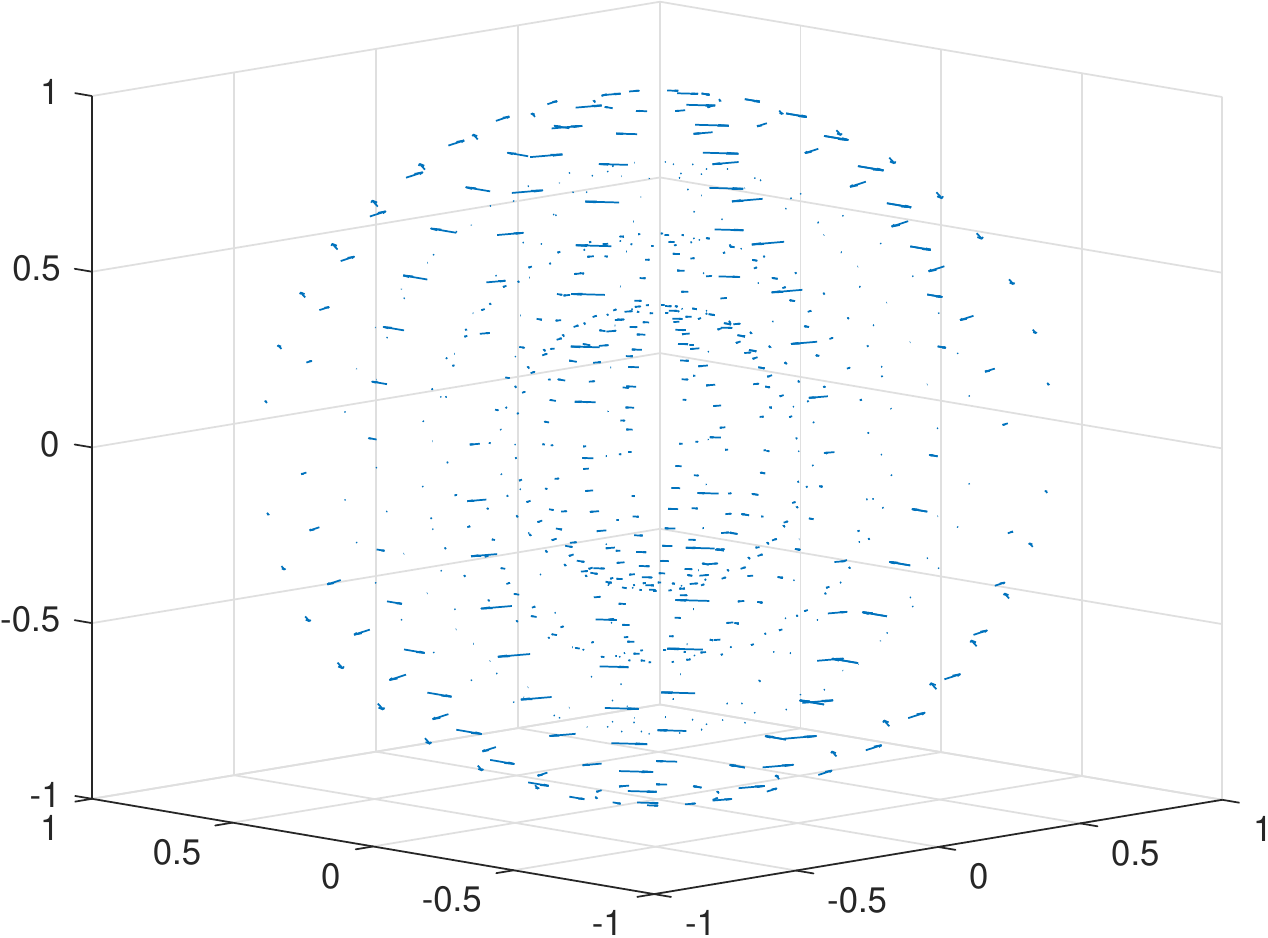}}
  \caption{Eigenfunctions  ${\bs\psi}_{k,\ell}^{\af,n}(\bx;c)$ with $c=10$.}
\end{figure}

\begin{appendix}
\renewcommand{\theequation}{A.\arabic{equation}}

\section{Proofs of Lemma \ref{calculus} and Lemma \ref{calculus2}}
\label{AppFirst}
\begin{proof}[Proof of Lemma \ref{calculus}]
Since $r$ and $\nabla_0$ commute,  one has
 \begin{align*}
  \nabla_0 \hat {\bs x}^{\tr}  = \frac{1}{r} \nabla_0 {\bs x}^{\tr}
  \overset{\eqref{opert-1b}} = \nabla  {\bs x}^{\tr}  -  \hat {\bs x} \partial_r  {\bs x}^{\tr}
  =   {\bs I}  -  \hat {\bs x}  \hat {\bs x} ^{\tr},
 \end{align*}
 and
\begin{align*}\hat {\bs x}\times \nabla_0 = {\bs x}\times \frac{1}{r} \nabla_0 \overset{\eqref{opert-1b}} =  {\bs x}\times (\nabla- \hat {\bs x} \partial_r)
=  {\bs x}\times \nabla =  -\nabla \times {\bs x},
\end{align*}
 which  verifies  
 \eqref{opert-1D}.

In analogy to   $ {\bs a} \cdot ({\bs a}\times {\bs b}) =  {\bs b} \cdot ({\bs a}\times {\bs b})=0$, one readily finds
\begin{align*}
\nabla \cdot( {\bs x}\times \nabla)  
= {\bs x} \cdot( {\bs x}\times \nabla)  = 0,
\end{align*}
which  proves \eqref{opert-1d}.

By a technical reduction, one gets
 \begin{align*}
 &({\bs x}\times \nabla)\cdot  ({\bs x}\times \nabla) =  \sum_{1\le i<j\le 3} (x_i\partial_{x_j}-x_j\partial_{x_j})^2
 = \|\bx\|^2 \Delta - (\bx\cdot \nabla )^2 -\bx\cdot \nabla \overset{\eqref{Lap-Bel}} = \Delta_0,
 \end{align*}
 which shows \eqref{opert-1j}.

 It is obvious that
\begin{align}
&(\bx \cdot \nabla )\, (\bx \times \nabla ) \overset{\eqref{opert-1D}}= r\partial_r  (\hat \bx \times \nabla_0 )
= (\hat \bx \times \nabla_0 )r\partial_r  \overset{\eqref{opert-1D}}=    (\bx \times \nabla )  \times (\bx \cdot \nabla ),
\label{commute1}
\\
&\|\bx \|^2  (\bx \times \nabla )  \overset{\eqref{opert-1D}}=r^2  (\hat \bx \times \nabla_0 ) =  (\hat \bx \times \nabla_0 )r^2
\overset{\eqref{opert-1D}}= (\bx \times \nabla ) \|\bx \|^2.
\label{commute2}
\end{align}
Meanwhile, It is straight forward that
\begin{align*}
& \Delta  (x_i \partial_{x_j} - x_j \partial_{x_i} )  = [(x_i \partial_{x_j})  \Delta +  \partial_{x_i} \partial_{x_j}]
  - [(x_j \partial_{x_i})  \Delta +  \partial_{x_j} \partial_{x_i}] = (x_i \partial_{x_j}-x_j \partial_{x_i}) \Delta,
\end{align*}
which shows
 \begin{align}
 \label{commute3}
& \Delta  (\bx \times \nabla )= (\bx \times \nabla )\Delta  .
\end{align}
As a result,
\begin{align*}
\Delta_0 (\bx \cdot \nabla ) \overset{\eqref{Lap-Bel}}= &\,\|\bx\|^2  [\Delta -(\bx \cdot  \nabla )(\bx \cdot \nabla+1 )  ] (\bx \times \nabla )
\\
=\ \ &  (\bx \times \nabla )  \|\bx\|^2  [\Delta -(\bx \cdot  \nabla )(\bx \cdot \nabla+1 )  ]
\overset{\eqref{Lap-Bel}}=(\bx \cdot \nabla )\Delta_0,
\end{align*}
which reveals \eqref{D0xn}.

The proof is now completed.
\end{proof}

\begin{proof}[Proof of Lemma \ref{calculus2}]
The well-known identity  on the cross products ${\bs a}\times {\bs b}\times {\bs c} ={\bs b} ({\bs a} \cdot  {\bs c})- {\bs c}({\bs a} \cdot {\bs b} ) $
yields
\begin{align*}
&\bx \times ({\bs x}\times \nabla ) ={\bx} ({\bs x} \cdot \nabla )- \|{\bs x}\|^2 \, \nabla \overset{\eqref{opert-1b}}= -r \nabla_0,
\end{align*}
which gives  \eqref{opert-1g}.  In return,
\begin{align*}
\nabla \cdot [\bx \times( {\bs x}\times \nabla)]  \overset{\eqref{opert-1g}}=  -\nabla \cdot r \nabla_0  \overset{\eqref{opert-1b}}= - (\frac1r \nabla_0+\hat{\bs x} \partial_r) \cdot r \nabla_0   \overset{\eqref{opert-1a}}= -\Delta_0 ,
\end{align*}
which is exactly \eqref{opert-1h}.

 Applying the  identity  on the  double curl operator $\nabla \times \nabla  \times {\bs f} =\nabla (\nabla  \cdot  {\bs f})- \Delta {\bs f} $, one obtains
 \begin{align*}
 \begin{split}
\nabla  \times&( {\bs x}\times \nabla) = -\nabla \times \nabla\times  {\bs x}  =  \Delta {\bx}- \nabla \Div  {\bs x}
 \\
 =&\, {\bx} \Delta  +2 \nabla  - ({\bx}  \cdot \nabla+4) \nabla
  = {\bx} \Delta   - ({\bx}  \cdot \nabla+2) \nabla ,
 \end{split}
 \end{align*}
 which states \eqref{opert-1E}. In the sequel,
 \begin{align*}
\Div    \nabla \times {\bs x}\times \nabla = \Div   {\bx} \Delta   - \Div     ({\bx}  \cdot \nabla+2) \nabla
=  ( {\bx}\cdot \nabla+3 \Delta  ) \Delta -  ({\bx}  \cdot \nabla+3) \Delta =0,
 \end{align*}
 and
 \begin{align*}
{\bx} \cdot   \nabla \times {\bs x}\times \nabla = {\bx} \cdot   {\bx} \Delta   - {\bx} \cdot    ({\bx}  \cdot \nabla+2) \nabla
=  \| \bx \|^2  \Delta   - {\bx} \cdot  \nabla  ({\bx}  \cdot \nabla+1) \overset{\eqref{Lap-Bel}} = \Delta_0,
 \end{align*}
 which gives \eqref{opert-1i} and \eqref{opert-1c}, respectively.

This  ends the proof.
\end{proof}

\renewcommand{\theequation}{B.\arabic{equation}}

%

\end{appendix}

\bibliographystyle{plain}
\bibliography{vball}

\begin{thebibliography}{10}

\bibitem{Bar85}
R.~G. Barrera, G.~A. Estevez, and J.~Giraldo.
\newblock Vector spherical harmonics and their application to magnetostatics.
\newblock {\em European Journal of Physics}, 6:287--294, 1985.

\bibitem{Bonami.17}
A.~Bonami and A.~Karoui.
\newblock Approximations in sobolev spaces by prolate spheroidal wave
  functions.
\newblock {\em Appl. Comput. Harmon. Anal.}, 42(3):361--377, 2017.

\bibitem{Botezatu.16}
M.~Botezatu, H.~Hult, T.M. Kassaye, and U.G. Fors.
\newblock Generalized prolate spheroidal wave functions: spectral analysis and
  approximation of almost band-limited functions.
\newblock {\em Appl. J. Fourier Anal. Appl.}, 22(2):383--412, 2016.

\bibitem{Boyed.04}
J.P. Boyd.
\newblock Prolate spheroidal wavefunctions as an alternative to chebyshev and
  legendre polynomials for spectral element and pseudospectral algorithms.
\newblock {\em J. Comput. Phys.}, 199(2):688--716, 2004.

\bibitem{Boyd.acm}
J.P. Boyd.
\newblock Algorithm 840: computation of grid points, quadrature weights and
  derivatives for spectral element methods using prolate spheroidal wave
  functions---prolate elements.
\newblock {\em ACM Trans. Math. Software}, 31(1):149--165, 2005.

\bibitem{Brenner.07}
Susanne~C. Brenner, F.~Li, and L.Y. Sung.
\newblock A locally divergence-free nonconforming finite element method for the
  time-harmonic maxwell equations.
\newblock {\em Math. Comp.}, 76(258):573--595, 2007.

\bibitem{Chen.05}
Q.Y. Chen, D.~Gottlieb, and Hesthaven. J.
\newblock Spectral methods based on prolate spheroidal wave functions for
  hyperbolic pdes.
\newblock {\em SIAM J. Numer. Anal.}, 43(5):1912--1933, 2005.

\bibitem{Cockburn.05}
B.~Cockburn, F.~Li, and C.W. Shu.
\newblock Locally divergence-free discontinuous galerkin methods for the
  maxwell equations.
\newblock {\em J. Comput. Phys.}, 194:588--610, 2004.

\bibitem{Dai2013}
F.~Dai and Y.~Xu.
\newblock {\em Approximation Theory and Harmonic Analysis on Spheres and
  Balls}.
\newblock Springer-Verlag, 2013.

\bibitem{Jackson.91}
J.I. Jackson, C.H. Meyer, D.G. Nishimura, and A.~Macovski.
\newblock Selection of a convolution function for fourier inversion using
  gridding.
\newblock {\em IEEE Trans. Med. Imag.}, 10(3):473--478, 1991.

\bibitem{Khalid.16}
Z.~Khalid, R.~A. Kennedy, and J.~D. McEwen.
\newblock Slepian spatial-spectral concentration on the ball.
\newblock {\em Applied and Computational Harmonic Analysis}, 40(3):470--504,
  2016.

\bibitem{Rokhlin.12}
W.Y. Kong and V.~Rokhlin.
\newblock A new class of highly accurate differentiation schemes based on the
  prolate spheroidal wave functions.
\newblock {\em Appl. Comput. Harmon. Anal.}, 33(2):226--260, 2012.

\bibitem{Landa.17}
B.~Landa and Y.~Shkolnisky.
\newblock Approximation scheme for essentially bandlimited and
  space-concentrated functions on a disk.
\newblock {\em Appl. Computat. Harmon. Anal.}, 43(3):381--403, 2017.

\bibitem{Li.05}
Fengyan Li and Chi-Wang Shu.
\newblock Locally divergence-free discontinuous galerkin methods for {MHD}
  equations.
\newblock {\em Journal of Scientific Computing}, 22-23(1-3):413--442, 2005.

\bibitem{Mathews.16}
J.~Mathews, J.~Breakall, and G.~Karawas.
\newblock The discrete prolate spheroidal filter as a digital signal processing
  tool.
\newblock {\em IEEE Trans. Acoust. Speech Signal Process.}, 33(6):1471--1478,
  1985.

\bibitem{Moo.M04}
I.C. Moore and M.~Cada.
\newblock Prolate spheroidal wave functions, an introduction to the {S}lepian
  series and its properties.
\newblock {\em Appl. Comput. Harmon. Anal.}, 16(3):208--230, 2004.

\bibitem{Rokhlin.07}
V~Rokhlin and H.~Xiao.
\newblock Approximate formulae for certain prolate spheroidal wave functions
  valid for large values of both order and band-limit.
\newblock {\em Appl. Comput. Harmon. Anal.}, 22(1):105--123, 2007.

\bibitem{STW11}
J.~Shen, T.~Tang, and L.L. Wang.
\newblock {\em Spectral Methods: Algorithms, Analysis and Applications}.
\newblock Springer, 2011.

\bibitem{Simons.11}
F.~J. Simons and D.V. Wang.
\newblock Spatiospectral concentration in the cartesian plane.
\newblock {\em GEM - International Journal on Geomathematics}, 2(1):1--36,
  2011.

\bibitem{Slep61}
D.~Slepian and H.O. Pollak.
\newblock Prolate spheroidal wave functions, {F}ourier analysis and
  uncertainty. {I}.
\newblock {\em Bell System Tech. J.}, 40:43--63, 1961.

\bibitem{Thomson.76}
D.J. Thomson, M.F. Robbins, C.G. Maclennan, and L.J. Lanzerotti.
\newblock Spectral and windowing techniques in power spectral analyses of
  geomagnetic data.
\newblock {\em Phys. Earth Planet. Inter.}, 12(23):217--231, 1976.

\bibitem{wang2009new}
L.L. Wang and J.~Zhang.
\newblock {A new generalization of the PSWFs with applications to spectral
  approximations on quasi-uniform grids}.
\newblock {\em Appl. Comput. Harmon. Anal.}, 29(3):303--329, 2010.

\bibitem{ZhangLi.18}
J.~Zhang, H.~Li, L.~L. Wang, and Z.~Zhang.
\newblock Ball prolate spheroidal wave functions in arbitrary dimensions.
\newblock {\em Appl. Computat. Harmon. Anal. DOI:10.1016/j.acha.2018.08.001},
  2018.

\bibitem{ZhangWang.17}
J.~Zhang, L.~L. Wang, H.~Li, and Z.~Zhang.
\newblock Optimal spectral schemes based on generalized prolate spheroidal wave
  functions of order -1.
\newblock {\em J.Sci.Comput.}, 70:451--477, 2017.

\end{thebibliography}

\end{document}